\documentclass[12pt]{amsart}
\usepackage{amsmath,amscd,amsfonts,accents,amssymb,amsthm,a4wide,url,mathscinet,mathdots,bm,caption}

\newtheorem{theorem}{Theorem}[section]
\newtheorem{lemma}[theorem]{Lemma}
\newtheorem{definition}[theorem]{Definition}
\newtheorem{conjecture}[theorem]{Conjecture}
\newtheorem{example}[theorem]{Example}
\newtheorem{proposition}[theorem]{Proposition}
\newtheorem{corollary}[theorem]{Corollary}
\theoremstyle{remark}
\newtheorem{remark}[theorem]{Remark}

\begin{document}

\newcommand{\C}{\mathbb{C}}
\newcommand{\maskx}{{\underline{x}}}
\newcommand{\bs}{\backslash}
\newcommand{\dfct}{\mathfrak{d}}
\newcommand{\dfcts}{\mathcal{D}}
\newcommand{\ndfcts}{\mathcal{E}}
\newcommand{\duble}{\widetilde}
\newcommand{\even}{\operatorname{even}}
\newcommand{\odd}{\operatorname{odd}}
\newcommand{\omegaw}{\bm{\omega}}
\newcommand{\power}{\mathcal{P}}
\newcommand{\wordd}{\mathbf{w}}
\newcommand{\graph}{\mathcal{G}}
\newcommand{\s}{\mathbf{s}}
\newcommand{\id}{\operatorname{id}}
\newcommand{\rest}{\big|}
\newcommand{\Quadr}{\mathcal{T}}
\newcommand{\red}{\mathcal{R}^H}
\newcommand{\sgn}{\operatorname{sgn}}
\newcommand{\prm}{\mathbf{p}}
\newcommand{\iprm}{\mathbf{q}}
\newcommand{\rev}{\mathfrak{r}}
\newcommand{\lshft}[1]{\accentset{\leftarrow}{#1}}               
\newcommand{\rshft}[1]{\accentset{\rightarrow}{#1}}              
\newcommand{\abs}[1]{\left|{#1}\right|}
\newcommand{\Dwordd}{\duble\wordd}
\newcommand{\card}[1]{\left\vert{#1}\right\vert}
\newcommand{\supp}{\operatorname{supp}}
\newcommand{\GL}{\operatorname{GL}}

\title{Conjectures about certain parabolic Kazhdan--Lusztig polynomials}
\author{Erez Lapid}
\address{Department of Mathematics, Weizmann Institute of Science, Rehovot 7610001 Israel}
\email{erez.m.lapid@gmail.com}

\begin{abstract}
Irreducibility results for parabolic induction of representations of the general linear group over a local non-archimedean
field can be formulated in terms of Kazhdan--Lusztig polynomials of type $A$.
Spurred by these results and some computer calculations, we conjecture that certain alternating sums of Kazhdan--Lusztig polynomials known
as parabolic Kazhdan--Lusztig polynomials satisfy properties analogous to those of the ordinary ones.
\end{abstract}

\maketitle

\setcounter{tocdepth}{1}
\tableofcontents

\section{Introduction}

Let $P_{u,w}$ be the Kazhdan--Lusztig polynomials with respect to the symmetric groups $S_r$, $r\ge1$.
Recall that $P_{u,w}=0$ unless $u\le w$ in the Bruhat order.
Fix $m,n\ge1$ and let $H\simeq S_m\times\dots\times S_m$ be the parabolic subgroup of $S_{mn}$
of type $(m,\dots,m)$ ($n$ times).
(In the body of the paper we will only consider the case $m=2$, but for the introduction $m$ is arbitrary.)
The normalizer of $H$ in $S_{mn}$ is $H\rtimes N$ where
\[
N=\{\duble w:w\in S_n\}\text{ and }\duble w(mi-j)=mw(i)-j,\ i=1,\dots,n,\ j=0,\dots,m-1.
\]
As a consequence of representation-theoretic results, it was proved in \cite{1605.08545}
that if $x,w\in S_n$ with $x\le w$ and there exists $v\le x$ such that $P_{v,w}=1$ and $v$ is $(213)$-avoiding
(i.e., there do not exist indices $1\le i<j<k\le n$ such that $v(j)<v(i)<v(k)$) then
\begin{equation} \label{eq: idat1}
\sum_{u\in H}\sgn u\ P_{\duble xu,\duble w}(1)=1.
\end{equation}
(We refer the reader to \cite{1605.08545} for more details. The representation-theoretic context will not play
an active role in the current paper.)
Motivated by this result, we carried out some computer calculations which suggest the following conjectural refinement.
\begin{conjecture} \label{conj: main}
For any $x,w\in S_n$ with $x\le w$ write
\begin{equation} \label{eq: mainconj}
\sum_{u\in H}\sgn u\ P_{\duble xu,\duble w}=q^{{m\choose 2}(\ell(w)-\ell(x))}\tilde P^{(m)}_{x,w}.
\end{equation}
Then
\begin{enumerate}
\item $\tilde P^{(m)}_{x,w}$ is a polynomial (rather than a Laurent polynomial).
\item $\tilde P^{(m)}_{x,w}(0)=1$.
\item $\tilde P^{(m)}_{x,w}=\tilde P^{(m)}_{xs,w}$ for any simple reflection $s$ of $S_n$ such that $ws<w$.
\item $\deg\tilde P^{(m)}_{x,w}=m\deg P_{x,w}$. In particular, $\tilde P^{(m)}_{x,w}=1$ if and only if $P_{x,w}=1$.
\end{enumerate}
\end{conjecture}

\begin{remark} \label{rem: basicfacts}
\begin{enumerate}
\item The left-hand side of \eqref{eq: mainconj} is an instance of a parabolic Kazhdan--Lusztig polynomial
in the sense of Deodhar \cite{MR916182}. They are known to have non-negative coefficients \cite{MR1901161}
(see also \cite{MR2480718, MR3003920}).\footnote{This is now known for any Coxeter group and a parabolic subgroup thereof by Libedinsky--Williamson \cite{1702.00459}}
Thus, the same holds for $\tilde P^{(m)}_{x,w}$.
In particular, if \eqref{eq: idat1} holds (i.e., if $\tilde P^{(m)}_{x,w}(1)=1$) then the left-hand side of \eqref{eq: mainconj} is a priori a monomial (with coefficient $1$),
and the conjecture would say in this case that its degree is ${m\choose 2}(\ell(w)-\ell(x))$.
\item Conjecture \ref{conj: main} is trivially true for $m=1$ (in which case $\tilde P^{(m)}_{x,w}=P_{x,w}$),
or if $w=x$ (in which case $\tilde P^{(m)}_{x,w}=P_{x,w}=1$).
\item The summands on the left-hand side of \eqref{eq: mainconj} are $0$ unless $x\le w$.
\item If $w=w_1\oplus w_2$ (direct sum of permutations) then $x=x_1\oplus x_2$ with $x_i\le w_i$, $i=1,2$
and $\tilde P^{(m)}_{x,w}=\tilde P^{(m)}_{x_1,w_1}\tilde P^{(m)}_{x_2,w_2}$.
Thus, in Conjecture \ref{conj: main} we may assume without loss of generality that $w$ is indecomposable (i.e.,
does not belong to a proper parabolic subgroup of $S_n$).
\item The relations $\tilde P^{(m)}_{x,w}=\tilde P^{(m)}_{x^{-1},w^{-1}}=
\tilde P^{(m)}_{w_0xw_0,w_0ww_0}$ for the longest $w_0\in S_n$ are immediate from the definition and the analogous relations for $m=1$.
It is also not difficult to see that, just as in the case $m=1$, we have $\tilde P^{(m)}_{xs,ws}=\tilde P^{(m)}_{x,w}$
(and more precisely, $P_{u\tilde x\tilde s,\tilde w\tilde s}=P_{u\tilde x,\tilde w}$ for every $u\in H$)
for any $x,w\in S_n$ and a simple reflection $s$ such that $xs<x\not< ws<w$.
On the other hand, the third part of the conjecture does not seem to be a formal consequence of the analogous relation for $m=1$.
\item \label{rem: cancelable}
An index $i=1,\dots,n$ is called cancelable for $(w,x)$ if $w(i)=x(i)$ and $\#\{j<i:x(j)<x(i)\}=\#\{j<i:w(j)<w(i)\}$.
It is known that in this case $P_{x,w}=P_{x^i,w^i}$ where $w^i=\Delta_{w(i)}^{-1}\circ w\circ\Delta_i,x^i=\Delta_{x(i)}^{-1}\circ x\circ\Delta_i\in S_{n-1}$ and
$\Delta_j:\{1,\dots,n-1\}\rightarrow\{1,\dots,n\}\setminus\{j\}$ is the monotone bijection (see \cite{MR1990570, MR2320806}).
Clearly, if $i$ is cancelable for $(w,x)$ then $mi,mi-1,\dots,m(i-1)+1$ are cancelable for $(\tilde w,\tilde x)$,
and hence it is easy to see from the definition that $\tilde P^{(m)}_{x,w}=\tilde P^{(m)}_{x^i,w^i}$.
\item For $n=2$ (and any $m$) Conjecture \ref{conj: main} is a special case of a result of Brenti \cite{MR1972246}.
(See also \cite{MR3558217}.)
\end{enumerate}
\end{remark}

We verified the conjecture numerically for the cases where $nm\le12$.
In the appendix we provide all non-trivial $\tilde P^{(m)}_{x,w}$ in these cases.
In general, already for $m=2$, $\tilde P^{(m)}_{x,w}$ does not depend only on $P_{x,w}$.
Nevertheless, there seems to be some correlation between $\tilde P^{(2)}_{x,w}$ and $((P_{x,w})^2+P_{x,w}(q^2))/2$.

The purpose of this paper is to provide modest theoretical evidence, or a sanity check, for Conjecture \ref{conj: main}
in the case $m=2$.
Namely, we prove it in the very special case that $w$ is any Coxeter element of $S_n$ (or a parabolic subgroup thereof).
Note that $P_{e,w}=1$ in this case and the conjecture predicts that $\tilde P^{(2)}_{x,w}=1$ for any $x\le w$.
Following Deodhar \cite{MR1065215}, the assumption on $w$ guarantees that $P_{u,\duble w}$ admits a simple combinatorial formula for any $u\in S_{2n}$.
(This is a special case of a result of Billey--Warrington \cite{MR1826948} but the case at hand is much simpler.)
Thus, the problem becomes elementary. (For an analogous result in a different setup see \cite{MR3159260}.)

In principle, the method can also be used to prove Conjecture \ref{conj: main} for $m=3$ in the case where $w$
is the right or left cyclic shift. However, we will not carry this out here. Unfortunately, for $m>3$ the method is not applicable
for any $w\ne e$ (again, by the aforementioned result of \cite{MR1826948}).

In the general case, for instance if $w$ is the longest Weyl element, we are unaware of a simple combinatorial formula
for the individual Kazhdan--Lusztig polynomials $P_{u,\duble w}$, even for $m=2$.
Thus, Conjecture \ref{conj: main} becomes more challenging
and at the moment we do not have any concrete approach to attack it beyond the cases described above.
In particular, we do not have any theoretical result supporting the last part of the conjecture, which rests on thin air.

We mention that the relation \eqref{eq: idat1} admits the following generalization.
Suppose that $v,w\in S_n$, $P_{v,w}=1$ and $v$ is $(213)$-avoiding. Then
\begin{equation} \label{eq: KLid}
\sum_{u\in HxH}\sgn u\ P_{u,\duble w}(1)=\sum_{u\in HxH\cap K}\sgn u
\end{equation}
for any $x\in S_{2n}$ such that $\duble{v}\le x\le\duble w$ (\cite[Theorem 10.11]{1605.08545}, which uses \cite{1710.06115}).
Here $K$ is the subgroup of $S_{mn}$ (isomorphic to $S_n\times\dots\times S_n$, $m$ times) that preserves the congruence classes modulo $m$.
In the case $m=2$, $\sgn$ is constant on $HxH\cap K$ and the cardinality of $HxH\cap K$ is a power of $2$
that can be easily explicated in terms of $x$ (\cite[Lemma 10.6]{1605.08545}).
(For $m>2$ this is no longer the case. For instance, for $m=n$ and a suitable $x$, \eqref{eq: KLid} is $(-1)^{n\choose 2}$ times the difference between
the number of even and odd Latin squares of size $n\times n$.)
In general, already for $m=2$, the assumption that $v$ is $(213)$-avoiding is essential for \eqref{eq: KLid}
(in contrast to \eqref{eq: idat1} if Conjecture \ref{conj: main} holds).
For $m=2$ and Coxeter elements $w$ we give in Corollary \ref{cor: maincor} a simple expression for $\sum_{u\in HxH}\sgn u\ P_{u,\duble w}$
for any $x\in S_{2n}$. However, at the moment we do not know how to extend it, even conjecturally, to other $w$'s such that $P_{e,w}=1$.

It would be interesting to know whether Conjecture \ref{conj: main} admits a representation-theoretic interpretation.

Normalizers of parabolic subgroups of Coxeter groups were studied in \cite{MR0453885, MR576184, MR1654763, MR1688445}.
In particular, they are the semidirect product of the parabolic subgroup by a
complementary subgroup, which in certain cases is a Coxeter groups by itself.
It is natural to ask whether Conjecture \ref{conj: main} extends to other classes of parabolic Kazhdan--Lusztig
polynomials with respect to (certain) pairs of elements of the normalizer
(e.g., as in the setup of \cite[\S25.1]{MR1974442}).
Note that already for the Weyl group of type $B_2$ these parabolic Kazhdan--Lusztig
polynomials may vanish, so that a straightforward generalization of Conjecture \ref{conj: main} is too naive.
Nonetheless, recent results of Brenti, Mongelli and Sentinelli \cite{MR3497995}
(albeit rather special) may suggest that \emph{some} generalization (which is yet to be formulated) is not hopeless.
Perhaps there is even a deeper relationship between parabolic Kazhdan--Lusztig polynomials pertaining to different data (including ordinary ones).
(See \cite{MR3166061} for another result in this direction.) At the moment it is not clear what is the scope of such
a hypothetical relationship.

Another natural and equally important question, to which I do not have an answer, is whether the
geometric interpretations of the parabolic Kazhdan--Lusztig polynomials \cite{MR1901161, MR2480718, MR3003920, 1702.00459}
shed any light on Conjecture \ref{conj: main} or its possible generalizations.

\subsection*{Acknowledgement}
The author would like to thank Karim Adiprasito, Joseph Bernstein, Sara Billey, David Kazhdan, George Lusztig,
Greg Warrington. Geordie Williamson and Zhiwei Yun for helpful correspondence.
We also thank the referee for useful suggestions.

\section{A result of Deodhar}

\subsection{}
For this subsection only let $G$ be an algebraic semisimple group over $\C$ of rank $r$, $B$ a Borel subgroup of $G$
and $T$ a maximal torus contained in $B$.
We enumerate the simple roots as $\alpha_1,\dots,\alpha_r$, the corresponding simple reflexions by
$s_1,\dots,s_r$ and the corresponding minimal non-solvable parabolic subgroups by $Q_1,\dots,Q_r$.
Let $W$ be the Weyl group, generated by $s_1,\dots,s_r$.
The group $W$ is endowed with the Bruhat order $\le$, the length function $\ell$, and the sign character $\sgn:W\rightarrow\{\pm1\}$.

We consider words in the alphabet $\s_1,\s_2,\dots,\s_r$.
For any word $\wordd=\s_{j_1}\dots\s_{j_l}$ we write $\pi(\wordd)=s_{j_1}\dots s_{j_l}$ for the corresponding element in $W$
and $\supp(\wordd)=\{\alpha_{j_1},\dots,\alpha_{j_k}\}$.
We say that $\wordd$ is supported in $A$ if $A\supset\supp(\wordd)$.
We also write $\wordd^\rev$ for the reversed word $\s_{j_l}\dots\s_{j_1}$.

Let $\wordd=\s_{j_1}\dots\s_{j_l}$ be a reduced decomposition for $w\in W$ where $l=\ell(w)$.
The reversed word $\wordd^\rev$ is a reduced decomposition for $w^{-1}$.
Note that $\supp(\wordd)=\{\alpha_i:s_i\le w\}$ and in particular, $\supp(\wordd)$ depends only on $w$.
A $\wordd$-mask is simply a sequence of $l$ zeros and ones, i.e., an element of $\{0,1\}^l$.
For a $\wordd$-mask $\maskx\in\{0,1\}^l$ and $i=0,\dots,l$ we write $\wordd^{(i)}[\maskx]$
for the subword of $\wordd$ composed of the letters $\s_{j_k}$ for $k=1,\dots,i$ with $\maskx_k=1$.
For $i=l$ we simply write
\[
\wordd[\maskx]=\wordd^{(l)}[\maskx].
\]
Let
\begin{equation} \label{def: defect}
\dfcts_{\wordd}(\maskx)=\{i=1,\dots,l:\pi(\wordd^{(i-1)}[\maskx])(\alpha_{j_i})<0\},\ \ \dfct_{\wordd}(\maskx)=\card{\dfcts_{\wordd}(\maskx)}
\end{equation}
(the \emph{defect set} and \emph{defect statistics} of $\maskx$) where $\card{\cdot}$ denotes the cardinality of a set.
We also write
\begin{equation} \label{def: defect0}
\dfcts_{\wordd}^r(\maskx)=\{i\in\dfcts_{\wordd}(\maskx):\maskx_i=r\},\
\ndfcts_{\wordd}^r(\maskx)=\{i\notin\dfcts_{\wordd}(\maskx):\maskx_i=r\},\ r=0,1.
\end{equation}
Note that $\ell(\wordd[\maskx])=\card{\ndfcts_{\wordd}^1(\maskx)}-\card{\dfcts_{\wordd}^1(\maskx)}$ for any $\maskx$.
We say that $\maskx$ is full if $\maskx_i=1$ for all $i$.

For later use we also set
\[
\sgn\maskx=\prod_{i=1}^l(-1)^{\maskx_i}\in\{\pm1\}
\]
so that $\sgn\wordd[\maskx]=\sgn\maskx$.

It is well known that
\begin{multline} \label{eq: bint}
\{\pi(\wordd[\maskx]):\maskx\in\{0,1\}^l\}=
\{\pi(\wordd[\maskx]):\maskx\in\{0,1\}^l\text{ and }\wordd[\maskx]\text{ is reduced}\}\\=\{u\in W:u\le w\}.
\end{multline}

For any $u\in W$ define the polynomial
\[
P^{\wordd}_u=\sum_{\maskx\in\{0,1\}^{\ell(w)}:\pi(\wordd[\maskx])=u}q^{\dfct_{\wordd}(\maskx)}.
\]

Let
\[
\phi_{\wordd}:(Q_{j_1}\times \dots\times Q_{j_l})/B^{l-1}\rightarrow\overline{BwB},\ \ \
(q_1,\dots,q_l)\mapsto q_1\dots q_l
\]
be (essentially) the Bott--Samelson resolution  \cite{MR0354697} where $B^{l-1}$ acts by
$(q_1,\dots,q_l)\cdot (b_1,\dots,b_{l-1})=(q_1b_1,b_1^{-1}q_2b_2,\dots,b_{l-2}^{-1}q_{l-1}b_{l-1},b_{l-1}^{-1}q_l)$.

\begin{remark}
\begin{enumerate}
\item It is easy to see that $P^{\wordd}_u$ has constant term $1$ if $u\le w$ (cf.~top of p.~161 in \cite{MR3027577}).
\item It follows from the Bia\l ynicki-Birula decomposition that $P^{\wordd}_u$ is the Poincare polynomial for
$\phi_{\wordd}^{-1}(BuB)$ (cf.~\cite[Proposotion 3.9]{MR1065215}, \cite[Proposition 5.12]{MR3027577}).
In particular, since the diagram
\[
\begin{CD}
(Q_{j_1}\times \dots\times Q_{j_l})/B^{l-1} @>{\phi_{\wordd}}>>  \overline{BwB}\\
@VV{(q_1,\dots,q_l)\mapsto(q_l^{-1},\dots,q_1^{-1})}V  @VV{g\mapsto g^{-1}}V\\
(Q_{j_l}\times \dots\times Q_{j_1})/B^{l-1} @>{\phi_{\wordd^\rev}}>> \overline{Bw^{-1}B}
\end{CD}
\]
is commutative we have $P^{\wordd^\rev}_{u^{-1}}=P^{\wordd}_u$, a fact which is not immediately clear from the definition since
in general $\dfct_{\wordd}(\maskx)\ne\dfct_{\wordd^\rev}(\maskx^\rev)$ where $\maskx^\rev$ denotes the reversed $\wordd^\rev$-mask $\maskx^\rev_i=\maskx_{l+1-i}$.

\item In general, $P_u^\wordd$ heavily depends on the choice of $\wordd$ unless $w$ has the property that all its reduced decompositions
are obtained from one another by repeatedly interchanging adjacent commuting simple reflexions, i.e., $w$ is fully commutative.
\end{enumerate}
\end{remark}

For $u,w\in W$ we denote by $P_{u,w}$ the Kazhdan--Lusztig polynomial with respect to $W$ \cite{MR560412}.
In particular, $P_{u,w}=0$ unless $u\le w$, in which case $P_{u,w}(0)=1$ and all coefficients of $P_{u,w}$ are non-negative.
(This holds in fact for any Coxeter group by a recent result of Elias-Williamson \cite{MR3245013}.)
We also have $P_{w,w}=1$, $\deg P_{u,w}\le\frac12(\ell(w)-\ell(u)-1)$ for any $u<w$ and
\begin{equation} \label{eq: Pinv}
P_{u^{-1},w^{-1}}=P_{u,w}.
\end{equation}
In general, even for the symmetric group, it seems that no ``elementary'' manifestly positive combinatorial formula for $P_{u,w}$ is known.
However, implementable combinatorial formulas to compute $P_{u,w}$ do exist
(see \cite{MR2133266} and the references therein).

The following is a consequence of the main result of \cite{MR1065215}.
See \cite{MR1278702, MR1826948, MR3027577} for more details.
\begin{theorem} \cite{MR1065215} \label{thm: Deodhar}
Let $\wordd$ be a reduced decomposition for $w\in W$.
Then the following conditions are equivalent.
\begin{enumerate}
\item $\deg P^{\wordd}_u\le\frac12(\ell(w)-\ell(u)-1)$ for any $u<w$.
\item For every non-full $\wordd$-mask $\maskx$ we have $\card{\ndfcts^0_{\wordd}(\maskx)}>\card{\dfcts^0_{\wordd}(\maskx)}$.
\item For every $u\in W$ we have $P_{u,w}=P^{\wordd}_u$.
\item The Bott--Samelson resolution $\phi_{\wordd}$ is small.
\end{enumerate}
In particular, under these conditions
\[
P_{u,w}(1)=\card{\{\maskx\in\{0,1\}^{\ell(w)}:\pi(\wordd[\maskx])=u\}}
\]
for any $u\in W$.
\end{theorem}

Following Lusztig \cite{MR1261904} and Fan--Green \cite{MR1441960} we say that $w\in W$ is tight
if it satisfies the conditions of Theorem \ref{thm: Deodhar}. (In fact, this condition is independent of the choice of $\wordd$.)

\section{Certain classes of permutations}
From now on we specialize to the symmetric group $S_n$ on $n$ letters, $n\ge1$.
We enumerate $\alpha_1,\dots,\alpha_{n-1}$ in the usual way.
Thus, $w\alpha_i>0$ if and only if $w(i)<w(i+1)$.
We normally write elements $w$ of $S_n$ as $(w(1)\, \dots\,w(n))$.

Given $x\in S_m$ and $w\in S_n$ we say that $w$ \emph{avoids} $x$
if there do not exist indices $1\le i_1<\dots<i_m\le n$ such that
\[
\forall 1\le j_1<j_2\le m\ w(i_{j_1})<w(i_{j_2})\iff x(j_1)<x(j_2).
\]
Equivalently, the $n\times n$-matrix $M_w$ representing $w$ does not admit $M_x$ as a minor.

There is a vast literature on pattern avoidance.
We will only mention two remarkable closely related general facts.
The first is that given $m$, there is an algorithm, due to Guillemot--Marx, to detect whether $w\in S_n$ is $x$-avoiding
whose running time is linear in $n$ \cite{MR3376367}. (If $m$ also varies then the problem is NP-complete \cite{MR1620935}.)
Secondly, if $C_n(x)$ denotes the number of $w\in S_n$ which are $x$-avoiding
then it was shown by Marcos--Tardos that the Stanley--Wilf limit $C(x)=\lim_{n\rightarrow\infty}C_n(x)^{1/n}$
exists and is finite \cite{MR2063960}, and as proved more recently by Fox, it is typically exponential in $m$ \cite{Fox}.

We recall several classes of pattern avoiding permutations.
First, consider the $(321)$-avoiding permutations, namely those for which there do not exist $1\le i<j<k\le n$
such that $w(i)>w(j)>w(k)$.
It is known that $w$ is $(321)$-avoiding if and only if it is fully commutative \cite{MR1241505}.
The number of $(321)$-avoiding permutations in $S_n$ is the Catalan number $C_n={2n\choose n}-{2n\choose n-1}$ -- a well-known result which goes back
at least 100 years ago to MacMahon \cite{MR2417935}.

We say that $w\in S_n$ is \emph{smooth} if it avoids the patterns $(4231)$ and $(3412)$.
By a result of Lakshmibai--Sandhya, $w$ is smooth if and only if the closure $\overline{Bw B}$ of the cell $Bw B$ in $\GL_n(\C)$
(where $B$ is the Borel subgroup of upper triangular matrices) is smooth \cite{MR1051089}.
(A generating function for the number of smooth permutations in $S_n$ in given in \cite{MR2376109}.)
It is also known that $w$ is smooth $\iff P_{e,w}=1\iff P_{u,w}=1$ for all $u\le w$ \cite{MR788771}.

The $(321)$-avoiding smooth permutations (i.e., the $(321)$ and $(3412)$-avoiding permutations) are precisely
the products of distinct simple reflexions  \cite{MR1603806, MR1417303}, i.e., the Coxeter elements of the parabolic subgroups of $S_n$.
They are characterized by the property that the Bruhat interval $\{x\in S_n:x\le w\}$ is a Boolean lattice, namely
the power set of $\{i:s_i\le w\}$ \cite{MR2333139}.
They are therefore called \emph{Boolean permutations}.
The number of Boolean permutations in $S_n$ is $F_{2n-1}$ where $F_m$ is the Fibonacci sequence \cite{MR1603806, MR1417303}.

In \cite{MR1826948} tight permutations were classified by Billey--Warrington in terms of pattern avoidance.
Namely, $w$ is tight if and only if it avoids the following five permutations
\begin{equation} \label{eq: 5perms}
(321)\in S_3, \ (46718235), (46781235), (56718234), (56781234)\in S_8.
\end{equation}
For the counting function of this class of permutations see \cite{MR2043806}.

\begin{remark}
In \cite{MR1354702}, Lascoux gave a simple, manifestly positive combinatorial formula for $P_{u,w}$ in the case where $w$ is $(3412)$-avoiding
(a property also known as \emph{co-vexillary}).
Note that a $(321)$-hexagon-avoiding permutation $w$ is co-vexillary if and only if it is a Boolean permutation, in which case $P_{u,w}=1$
for all $u\le w$.
\end{remark}

\section{The Defect} \label{sec: ddc}
Henceforth (except for Remark \ref{rem: m>2} below) we assume, in the notation of the introduction, that $m=2$.
Recall the group homomorphism
\[
\duble{ }:S_n\rightarrow S_{2n}
\]
given by
\[
\duble w(2i)=2w(i),\ \duble w(2i-1)=2w(i)-1,\ \ i=1,\dots,n.
\]
(Technically, $\duble{ }$ depends on $n$ but the latter will be hopefully clear from the context.)

We will use Theorem \ref{thm: Deodhar} to derive a simple expression for $P_{u,\duble w}$ where $w$ is a Boolean permutation.
(Note that if $e\ne w\in S_n$ then $\duble w$ is not co-vexillary. Thus, Lascoux's formula is not applicable.)

\begin{remark}
It is easy to see that $w$ is Boolean if and only if $\duble w$ satisfies the pattern avoidance conditions of
\cite{MR1826948}. Thus, it follows from [ibid.] that $\duble w$ is tight. However, we will give a self-contained proof of this fact
since this case is much simpler and in any case the ingredients are needed for the evaluation of $P_{u,\duble w}$.
\end{remark}

{\bf For the rest of the paper we fix a Boolean permutation $w\in S_n$ and a reduced decomposition $\wordd=\s_{j_1}\dots\s_{j_l}$ for $w$}
where $l=\ell(w)$ and $j_1,\dots,j_l\in\{1,\dots,n-1\}$ are distinct. The choice of $\wordd$ plays little role
and will often be suppressed from the notation.

For any $x\in S_n$ let
\begin{equation} \label{eq: Itau}
I_x=\{i:s_{j_i}\le x\}.
\end{equation}
Then
\begin{equation} \label{eq: booleanmap}
x\mapsto I_x\text{ is a bijection between }\{x\in S_n:x\le w\}\text{ and }\power(\{1,\dots,l\})
\end{equation}
where $\power$ denotes the power set.

A key role is played by the following simple combinatorial objects.
\begin{definition}
Let $A$ and $B$ be subsets of $\{1,\dots,l\}$ with $A\subset B$.
\begin{enumerate}
\item The right (resp., left) neighbor set $\rshft{N}_A^B=\,^{\wordd}\rshft{N}_A^B$ (resp., $\lshft{N}_A^B=\,^{\wordd}\lshft{N}_A^B$) of $A$
in $B$ with respect to $\wordd$ consists of the elements $i\in B\setminus A$ for which
there exists $t>0$ and indices $i_1,\dots,i_t$, necessarily unique, such that $i<i_1<\dots<i_t$
(resp., $i>i_1>\dots>i_t$), $\{i_1,\dots,i_{t-1}\}\subset A$, $i_t\in B\setminus A$ and $j_{i_k}=j_i+k$ for $k=1,\dots,t$.
\item The neighbor set of $A$ in $B$ with respect to $\wordd$ is $N_A^B=\rshft{N}_A^B\cup\lshft{N}_A^B$.
\item The neighboring function $\nu_A^B=\,^{\wordd}\nu_A^B:N_A^B\rightarrow B\setminus A$ is given by the rule $i\mapsto i_t$.
\end{enumerate}
\end{definition}
Note that the sets $\rshft N_A^B$ and $\lshft N_A^B$ are disjoint
and that $\nu_A^B$ is injective. Moreover, if $i\in N_A^B$ then $\nu_A^B(i)>i$ if and only if $i\in\rshft N_A^B$.

If $B=\{1,\dots,l\}$ then we suppress $B$ from the notation.
Note that
\begin{equation} \label{eq: NAB}
N_A^B=N_A\cap B\cap\nu_A^{-1}(B)\text{ and }\nu_A\rest_{N_A^B}=\nu_A^B.
\end{equation}

We have $\ell(\duble w)=4l$ and a reduced decomposition for $\duble w$ is given by
\[
\Dwordd=\s_{2j_1}\s_{2j_1-1}\s_{2j_1+1}\s_{2j_1}\dots\s_{2j_l}\s_{2j_l-1}\s_{2j_l+1}\s_{2j_l}.
\]
It will be convenient to write $\Dwordd$-masks as elements of $(\{0,1\}^4)^l$.
Thus, if $\maskx$ is a $\Dwordd$-mask then $\maskx_i\in\{0,1\}^4$, $i=1,\dots,l$ and we write $\maskx_{i,k}$, $k=1,2,3,4$
for the coordinates of $\maskx_i$.
By convention, we write for instance $\maskx_i=(*,1,*,0)$, to mean that $\maskx_{i,2}=1$ and $\maskx_{i,4}=0$, without restrictions
on $\maskx_{i,1}$ or $\maskx_{i,3}$.

For the rest of the section we fix a $\Dwordd$-mask $\maskx\in (\{0,1\}^4)^l$ and let
\[
I_f=\{i\in\{1,\dots,l\}:\maskx_i=(1,1,1,1)\}.
\]
We explicate the defect set $\dfcts_{\Dwordd}(\maskx)$ of $\maskx$ (see \eqref{def: defect}).

\begin{lemma} \label{lem: detdfct}
For any $i=1,\dots,l$ let $C(\maskx,i)$ (resp., $\lshft{C}(\maskx,i)$) be the condition
\[
i\in\lshft N_{I_f}\text{ and either }\maskx_{\nu_{I_f}(i)}=(1,1,0,1)\text{ or }\maskx_{\nu_{I_f}(i)}=(*,1,*,0)
\]
(resp.,
\[
\exists r\in\rshft N_{I_f}\text{ such that }\nu_{I_f}(r)=i\text{ and either }\maskx_r=(1,0,1,1)\text{ or }\maskx_r=(*,*,1,0)).
\]
Then for all $i=1,\dots,l$ we have
\begin{enumerate}
\item $\pi(\Dwordd^{(4i-4)}[\maskx])\alpha_{2j_i}>0$.
\item $\pi(\Dwordd^{(4i-3)}[\maskx])\alpha_{2j_i-1}<0$ if and only if $\maskx_i=(0,*,*,*)$ and $\lshft{C}(\maskx,i)$.
\item $\pi(\Dwordd^{(4i-2)}[\maskx])\alpha_{2j_i+1}<0$ if and only if $\maskx_i=(0,*,*,*)$ and $C(\maskx,i)$.
\item $\pi(\Dwordd^{(4i-1)}[\maskx])\alpha_{2j_i}<0$ if and only if (exactly) one of the following conditions is satisfied
\begin{gather*}
\maskx_i=(1,0,0,*),\\
\maskx_i=(1,0,1,*)\text{ and }C(\maskx,i),\\
\maskx_i=(1,1,0,*)\text{ and }\lshft C(\maskx,i).
\end{gather*}
\end{enumerate}
\end{lemma}

\begin{proof}
\begin{enumerate}
\item This is clear since $\alpha_{2j_i}\notin\supp(\Dwordd^{(4i-4)}[\maskx])$.
\item If $\maskx_{i,1}=1$ then $\pi(\Dwordd^{(4i-3)}[\maskx])\alpha_{2j_i-1}=\pi(\Dwordd^{(4i-4)}[\maskx])(\alpha_{2j_i-1}+\alpha_{2j_i})$
which as before is a positive root since $\alpha_{2j_i}\notin\supp(\Dwordd^{(4i-4)}[\maskx])$.

Suppose from now on that $\maskx_{i,1}=0$ and let $t\ge0$ be the largest index for which
there exist (unique) indices $i_t<\dots <i_1<i_0=i$ with $\{i_1,\dots,i_t\}\subset I_f$
such that $j_{i_k}=j_i-k$ for $k=1,\dots,t$.
If there does not exist $r<i_t$ such that $j_r=j_i-t-1$ then
$\pi(\Dwordd^{(4i-3)}[\maskx])\alpha_{2j_i-1}=\alpha_{2j_{i_t}-1}>0$.
Otherwise, by the maximality of $t$, we have $r\in\rshft N_{I_f}$, $\nu_{I_f}(r)=i$ and
$\pi(\Dwordd^{(4i-3)}[\maskx])\alpha_{2j_i-1}=\pi(\Dwordd^{(4r)}[\maskx])\alpha_{2j_r+1}$.
We split into cases.
\begin{enumerate}
\item If $\maskx_{r,3}=0$ then $\alpha_{2j_r+1}\notin\supp(\Dwordd^{(4r)}[\maskx])$ and therefore $\pi(\Dwordd^{(4i-3)}[\maskx])\alpha_{2j_i-1}>0$.
\item Assume that $\maskx_{r,3}=1$.
\begin{enumerate}
\item If $\maskx_{r,4}=0$ then
$\pi(\Dwordd^{(4r)}[\maskx])\alpha_{2j_r+1}=-\pi(\Dwordd^{(4r-2)}[\maskx])\alpha_{2j_r+1}<0$
since $\alpha_{2j_r+1}\notin\supp(\Dwordd^{(4r-2)}[\maskx])$.
\item Assume that $\maskx_{r,4}=1$, so that $\pi(\Dwordd^{(4r)}[\maskx])\alpha_{2j_r+1}=
\pi(\Dwordd^{(4r-2)}[\maskx])\alpha_{2j_r}$. The latter is a positive root unless
$\maskx_{r,1}=1$ (since otherwise $\alpha_{2j_r}\notin\supp(\Dwordd^{(4r-2)}[\maskx])$).
If $\maskx_{r,1}=1$ then $\maskx_{r,2}=0$ (since $r\notin I_f$)
and $\pi(\Dwordd^{(4i-3)}[\maskx])\alpha_{2j_i-1}=-\pi(\Dwordd^{(4r-4)}[\maskx])\alpha_{2j_r}<0$
since $\alpha_{2j_r}\notin\supp(\Dwordd^{(4r-4)}[\maskx])$.
\end{enumerate}
\end{enumerate}
\item This is similar to the second part. We omit the details.
\item If $\maskx_{i,1}=0$ then $\pi(\Dwordd^{(4i-1)}[\maskx])\alpha_{2j_i}>0$ since $\alpha_{2j_i}\notin\supp(\Dwordd^{(4i-1)}[\maskx])$.

Assume that $\maskx_{i,1}=1$. We split into cases.
\begin{enumerate}
\item If $\maskx_{i,2}=\maskx_{i,3}=0$ then $\pi(\Dwordd^{(4i-1)}[\maskx])\alpha_{2j_i}=-\pi(\Dwordd^{(4i-4)}[\maskx])\alpha_{2j_i}<0$
since $\alpha_{2j_i}\notin\supp(\Dwordd^{(4i-4)}[\maskx])$.
\item For the same reason, if $\maskx_{i,2}=\maskx_{i,3}=1$ then
$\pi(\Dwordd^{(4i-1)}[\maskx])\alpha_{2j_i}=\pi(\Dwordd^{(4i-4)}[\maskx])(\alpha_{2j_i}+\alpha_{2j_i-1}+\alpha_{2j_i+1})>0$.
\item If $\maskx_{i,2}=1$ and $\maskx_{i,3}=0$ then $\pi(\Dwordd^{(4i-1)}[\maskx])\alpha_{2j_i}=\pi(\Dwordd^{(4i-4)}[\maskx])\alpha_{2j_i-1}=
\pi(\Dwordd^{(4i-3)}[\maskx'])\alpha_{2j_i-1}$ where $\maskx'_j=\maskx_j$ for all $j\ne i$ and $\maskx'_{i,1}=0$.
This case was considered in the second part.
\item Similarly, the case $\maskx_{i,2}=0$ and $\maskx_{i,3}=1$ reduces to the third part.
\end{enumerate}
\end{enumerate}
\end{proof}

For a subset $A\subset\{1,\dots,l\}$ we denote by $A^c$ its complement in $\{1,\dots,l\}$.

\begin{corollary} \label{cor: changed}
Let $i\in(\rshft N_{I_f}\cup\nu_{I_f}(\lshft N_{I_f}))^c$.
\begin{enumerate}
\item Let $\maskx'\in(\{0,1\}^4)^l$ be such that $\maskx'_j=\maskx_j$ for all $j\ne i$ and
either $\maskx'_{i,1}=\maskx_{i,1}=0$ or $\maskx'_{i,1}=\maskx_{i,1}=1$, $\maskx'_{i,2}=\maskx_{i,2}$ and
$\maskx'_{i,3}=\maskx_{i,3}$. Then $\dfct_{\Dwordd}(\maskx')=\dfct_{\Dwordd}(\maskx)$.
\item Assume that $\maskx_{i,1}=1$.
Let $\wordd'$ be the word obtained from $\wordd$ by removing $\s_{j_i}$ and let
$\maskx'$ be the $\duble{\wordd'}$-mask obtained from $\maskx$ by removing $\maskx_i$.
Then
\[
\dfct_{\Dwordd}(\maskx)-\dfct_{\duble{\wordd'}}(\maskx')=\begin{cases}1&\maskx_{i,2}=\maskx_{i,3}=0,\\
\delta_1&\maskx_{i,2}=0,\maskx_{i,3}=1,\\\delta_2&\maskx_{i,2}=1,\maskx_{i,3}=0,\\
0&\maskx_{i,2}=\maskx_{i,3}=1,\end{cases}
\]
where $\delta_1$ (resp., $\delta_2$) is $1$ if $C(\maskx,i)$ (resp., $\lshft C(\maskx,i)$) holds and $0$ otherwise.
\end{enumerate}
\end{corollary}

\begin{corollary} \label{cor: dtight}
$\duble w$ is tight.
\end{corollary}

\begin{proof}
Assume that $I_f\ne\{1,\dots l\}$ i.e., $\maskx$ is not full.
We identify $\{1,\dots,l\}\times\{1,2,3,4\}$ with $\{1,\dots,4l\}$ by $(i,\delta)\mapsto 4(i-1)+\delta$.
For any $i\notin I_f$ let $m_i\in\{1,2,3,4\}$ be the smallest index such that $\maskx_{i,m_i}=0$.
It follows from Lemma \ref{lem: detdfct} that $(i,m_i)\in \ndfcts^0_{\Dwordd}(\maskx)$ (see \eqref{def: defect0}).
Define a map
\[
h_{\maskx}:\dfcts^0_{\Dwordd}(\maskx)\rightarrow\ndfcts_{\Dwordd}^0(\maskx)
\]
according to the rule
\[
h_\maskx((i,\delta))=\begin{cases}(i,1)&\delta=2,\\(\nu_{I_f}(i),3)&\delta>2,\ \maskx_{i,1}=\maskx_{i,3}\text{ and }\maskx_{\nu_{I_f}(i)}=(1,1,0,1),\\
(\nu_{I_f}(i),4)&\delta>2,\ \maskx_{i,1}=\maskx_{i,3}\text{ and }\maskx_{\nu_{I_f}(i)}=(*,1,*,0),\\
(i,3)&\delta=4\text{ and }\maskx_i=(1,*,0,0).
\end{cases}
\]
By Lemma \ref{lem: detdfct} $h_\maskx$ is well defined since
if $C(\maskx,i)$ is satisfied then $\lshft C(\maskx,\nu_{I_f}(i))$ is not satisfied.
Moreover, $h_{\maskx}$ is injective since $\nu_{I_f}$ is.
We claim that $h_\maskx$ is not onto.
Indeed, let $i\notin I_f$ be such that $j_i$ is minimal. Then $i\notin\nu_{I_f}(N_{I_f})$, and in particular
$\lshft C(\maskx,i)$ is not satisfied. Thus, $(i,\delta)$ is not in the image of $h_\maskx$ unless $\maskx_i=(1,0,0,0)$ and $\delta=3$.
Hence, $(i,m_i)\in\ndfcts^0_{\Dwordd}(\maskx)$ but $(i,m_i)$ is not in the image of $h_\maskx$.

It follows that
\[
\card{\ndfcts_{\Dwordd}^0(\maskx)}>\card{\dfcts^0_{\Dwordd}(\maskx)}.
\]
Since $\maskx$ is arbitrary (non-full), $\duble w$ is tight.
\end{proof}

\begin{remark} \label{rem: m>2}
Note that for $m=3$, $\duble w$ avoids the five permutations in \eqref{eq: 5perms} if and only if
$w$ avoids the patterns $(321)$, $(3412)$, $(3142)$, $(2413)$. It is easy to see that these $w$'s are exactly the permutations
which can be written as direct sums of (left or right) cyclic shifts. In principle, it should be possible to check
Conjecture \ref{conj: main} for $m=3$ for these permutations.
We will not provide any details here.
Note that for $m>3$ $\duble w$ does not avoid $(56781234)$ unless $w=e$ so this method fails.
\end{remark}

\section{Double cosets}
Let $H=H_n$ be the parabolic subgroup of $S_{2n}$ consisting of permutations which preserve each of the sets
$\{2i-1,2i\}$, $i=1,\dots,n$. Thus, $H$ is an elementary abelian group of order $2^n$.
Note that $H$ is normalized by $\duble{S_n}$.
It is well-known that the double cosets $H\bs S_{2n}/H$ are parameterized by $n\times n$-matrices
with non-negative integer entries, whose sums along each row and each column are all equal to $2$.
By the Birkhoff--von-Neumann Theorem, these matrices are precisely the sums of two $n\times n$-permutation matrices.

We denote by $\red$ the set of bi-$H$-reduced elements in $S_{2n}$, i.e.
\[
\red=\{w\in S_{2n}:w(2i)>w(2i-1)\text{ and }w^{-1}(2i)>w^{-1}(2i-1)\text{ for all }i=1,\dots,n\}.
\]
Each $H$-double coset contains a unique element of $\red$.

Our goal in this section is to parameterize the double cosets of $H$ containing an element $\le\duble w$,
or equivalently, the set $\red_{\le\duble w}:=\{u\in\red:u\le\duble w\}$.

\begin{definition}
Let $\Quadr=\Quadr^\wordd$ be the set of triplets $(I_e,I,I_f)$ of subsets of $\{1,\dots,l\}$ such that $I_f\subset I_e$
and $I\subset N_{I_f}^{I_e}$.
We will write $I=\rshft I\cup\lshft I$ (disjoint union) where $\rshft I=I\cap\rshft N_{I_f}$ and $\lshft I=I\cap\lshft N_{I_f}$.
\end{definition}

For any $u\le\duble w$ define $\prm(u)=\prm^\wordd(i)=(I_f,\rshft I\cup\lshft I,I_e)$ where
\begin{gather*}
I_e=\{i:s_{2j_i}\le u\},\ \ I_f=\{i:\duble{s_{j_i}}\le u\},\\
\rshft I=\{i\in\rshft N_{I_f}:s_{2j_i}s_{2j_i+1}\duble{s_{j_i+1}}\dots\duble{s_{j_{\nu_{I_f}(i)-1}}}s_{2j_{\nu_{I_f}(i)}}\le u\},\\
\lshft I=\{i\in\lshft N_{I_f}:s_{2j_{\nu_{I_f}(i)}}\duble{s_{j_{\nu_{I_f}(i)-1}}}\dots\duble{s_{j_i+1}}s_{2j_i+1}s_{2j_i}\le u\}.
\end{gather*}
Clearly, $\prm(u)\in\Quadr$ by \eqref{eq: NAB}. Note that if $i\in\rshft N_{I_f}$ (resp., $i\in\lshft N_{I_f}$) then
$s_{j_i+1}\dots s_{j_{\nu_{I_f}(i)-1}}$ (resp., $s_{j_{\nu_{I_f}(i)-1}}\dots s_{j_i+1}$) is the cyclic shift
\[
t\mapsto\begin{cases}t+1&j_i<t<j_{\nu_{I_f}(i)}\\j_i+1&t=j_{\nu_{I_f(i)}}\\t&\text{otherwise}\end{cases}\ \ \ \text{(resp., }
t\mapsto\begin{cases}t-1&j_i+1<t\le j_{\nu_{I_f}(i)}\\j_{\nu_{I_f(i)}}&t=j_i+1\\t&\text{otherwise}\end{cases}).
\]

Also,
\begin{equation} \label{eq: pdc}
\text{if $v\le u\le\duble w$ and $v\in HuH\cap\red$ then $\prm(v)=\prm(u)$.}
\end{equation}
Indeed, for any $y\in\red$ we have $y\le u$ if and only if $y\le v$.
Thus, $\prm$ is determined by its values on $\red_{\le\duble w}$.

In the other direction, consider the map
\[
\iprm=\iprm^\wordd:\Quadr\rightarrow S_{2n}
\]
given by $Q=(I_e,I,I_f)\mapsto\pi(\omegaw_Q)$ where
\begin{equation} \label{def: omegaQ}
\omegaw_Q=\mathbf{y}_1\dots\mathbf{y}_l,\ \ \mathbf{y}_i=\begin{cases}\emptyset&i\notin I_e,\\\duble{\s_{j_i}}=\s_{2j_i}\s_{2j_i-1}\s_{2j_i+1}\s_{2j_i}&i\in I_f,\\
\s_{2j_i}\s_{2j_i+1}&i\in \rshft{I},\\\s_{2j_i+1}\s_{2j_i}&i\in\lshft{I},\\\s_{2j_i}&\text{otherwise.}\end{cases}
\end{equation}
Clearly $\iprm(Q)\le\duble w$ for all $Q\in\Quadr$.
We also remark that
\begin{equation} \label{eq: qtau}
\iprm((I_x,\emptyset,I_x))=\duble x\text{ for all }x\le w
\end{equation}
(see \eqref{eq: Itau}).

\begin{proposition} \label{prop: paramdc}
The map $\prm$ is a bijection 
between $\red_{\le\duble w}$ and $\Quadr$ whose inverse is $\iprm$.
\end{proposition}

The proposition will follow from Lemmas \ref{lem: dblcst} and \ref{lem: descoset} below.

\begin{lemma} \label{lem: dblcst}
We have $\prm\circ\iprm=\id_{\Quadr}$. In particular, $\iprm$ is injective.
Moreover, the image of $\iprm$ is contained in $\red_{\le\duble w}$.
\end{lemma}

\begin{proof}
Let $Q=(I_e,I,I_f)\in\Quadr$.
We first claim that $\omegaw_Q$ is a reduced word.
Indeed, let
\[
\maskx_i=\begin{cases}(0,0,0,0)&i\notin I_e,\\(1,1,1,1)&i\in I_f,\\(1,0,1,0)&i\in\rshft I,\\(0,0,1,1)&i\in\lshft I,\\(0,0,0,1)&\text{otherwise,}\end{cases}
\]
so that $\omegaw_Q=\Dwordd[\maskx]$.
Then it is easy to see from Lemma \ref{lem: detdfct} that $\Dwordd[\maskx]$ is reduced i.e., that $\pi(\Dwordd^{(4i-k)}[\maskx])\alpha_{2j_i+t_k}>0$
whenever $i=1,\dots,l$ and $k=1,2,3,4$ are such that $\maskx_{i,k}=1$ where $t_1=t_4=0$, $t_2=1$, $t_3=-1$.
(The condition $C(\maskx,i)$ is never satisfied.)

Let us show that $\prm(\iprm(Q))=Q$. Write $\prm(\iprm(Q))=(I_f^\circ,\rshft I^\circ\cup\lshft I^\circ,I_e^\circ)$.
Since $\omegaw_Q$ is reduced, it is clear from the definition and from \eqref{eq: bint} that $I_e^\circ=I_e$, $I_f\subset I_f^\circ$,
$\rshft I\subset\rshft I^\circ$, $\lshft I\subset\lshft I^\circ$.
Since the only reduced decompositions of $\duble{s_{j_i}}$ are $\s_{2j_i}\s_{2j_i\pm 1}\s_{2j_i\mp 1}\s_{2j_i}$
(and in particular $\s_{2j_i}$ occurs twice) we must have $I_f=I_f^\circ$.
Let $i\in\rshft N_{I_f}$ and suppose that $v:=s_{2j_i}s_{2j_i+1}\duble{s_{j_i+1}}\dots\duble{s_{j_{\nu_{I_f}(i)-1}}}s_{2j_{\nu_{I_f}(i)}}\le\pi(\omegaw_Q)$.
Then $v$ is represented by a subword of $\omegaw_Q$. On the other hand, it is clear that any subword of $\omegaw_Q$
supported in $\{\s_k:2j_i\le k\le 2j_{\nu_{I_f}(i)}\}$ is a subword of
$\s_{2j_i}\s_{2j_i+1}\duble{\s_{j_i+1}}\dots\duble{\s_{j_{\nu_{I_f}(i)-1}}}\s_{2j_{\nu_{I_f}(i)}}$ and the latter is a subword of $\omegaw_Q$
only if $i\in\rshft I$. Hence, $\rshft I=\rshft I^\circ$. Similarly one shows that $\lshft I=\lshft I^\circ$.

Finally, let us show that $\iprm(Q)\in \red$. Let $u\in\red\cap H\iprm(Q)H$. Then $u\le\iprm(Q)$ and by \eqref{eq: pdc} and the above
we have $\prm(u)=\prm(\iprm(Q))=Q$.
It is easy to see that this is impossible unless $u=\iprm(Q)$.
\end{proof}

\begin{remark}
Define a partial order on $\Quadr$ by $Q_1=(I_e^{(1)},I^{(1)},I_f^{(1)})\le Q_2=(I_e^{(2)},I^{(2)},I_f^{(2)})$ if
$I_e^{(1)}\subset I_e^{(2)}$, $I_f^{(1)}\subset I_f^{(2)}$ and
$I^{(1)}\subset I^{(2)}\cup I_f^{(2)}\cup\nu_{I_f^{(1)}}^{-1}(I_f^{(2)})$.
Then it is not hard to check that $\iprm(Q_1)\le\iprm(Q_2)$ if and only if $Q_1\le Q_2$.
We omit the details since we will not use this fact.
\end{remark}

In order to complete the proof of Proposition \ref{prop: paramdc} we first record the following elementary assertion.

For any $i=1,\dots,l$ let $\rshft\mu_{\pm}(i)$ (resp., $\lshft\mu_{\pm}(i)$) be $i_t$ where $t\ge0$ is the largest index for which there exist (unique) indices $i=i_0<i_1<\dots<i_t\le l$
(resp., $i=i_0>i_1>\dots>i_t>0$) such that $j_{i_k}=j_i\pm k$ for $k=1,\dots,t$.

\begin{lemma} \label{lem: sm}
Let $\maskx,\maskx'\in(\{0,1\})^l$, $I_f=\{i=1,\dots,l:\maskx_i=(1,1,1,1)\}$ and $I_e=\{i=1,\dots,l:\maskx_i\ne(0,*,*,0),(1,0,0,1)\}$.
\begin{enumerate}
\item Suppose that $i$ is such that $\maskx_j=\maskx'_j$ for all $j\ne i$ and let $\epsilon_1,\epsilon_2\in\{0,1\}$.
\begin{enumerate}
\item \label{case: triv} If $\maskx_i=(1,0,0,1)$ and $\maskx'_i=(0,0,0,0)$ then $\pi(\Dwordd[\maskx])=\pi(\Dwordd[\maskx'])$.
\item \label{case: rshft} If $i\notin\rshft N_{I_f}\cap\nu_{I_f}^{-1}(I_e)$ and either $\maskx_i=(\epsilon_1,\epsilon_2,1,0)$, $\maskx'_i=(\epsilon_1,\epsilon_2,0,0)$
or $\maskx_i=(1,0,1,1)$, $\maskx'_i=(0,0,1,1)$ then $\pi(\Dwordd[\maskx])=\pi(\Dwordd[\maskx'])s_{2k+1}$
where $k=j_{\nu_{I_f}(i)}-1$ if $i\in\rshft N_{I_f}$ and $k=j_{\rshft\mu_+(i)}$ otherwise.
\item If $i\notin\lshft N_{I_f}\cap\nu_{I_f}^{-1}(I_e)$ and either $\maskx_i=(0,\epsilon_1,1,\epsilon_2)$, $\maskx'_i=(0,\epsilon_1,0,\epsilon_2)$
or $\maskx_i=(1,0,1,1)$, $\maskx'_i=(1,0,1,0)$ then $\pi(\Dwordd[\maskx])=s_{2k+1}\pi(\Dwordd[\maskx'])$
where $k=j_{\nu_{I_f}(i)}-1$ if $i\in\lshft N_{I_f}$ and $k=j_{\lshft\mu_+(i)}$ otherwise.
\item If $i\notin\nu_{I_f}(\lshft N_{I_f}\cap I_e)$ and either $\maskx_i=(\epsilon_1,1,\epsilon_2,0)$, $\maskx'_i=(\epsilon_1,0,\epsilon_2,0)$
or $\maskx_i=(1,1,0,1)$, $\maskx'_i=(0,1,0,1)$ then $\pi(\Dwordd[\maskx])=\pi(\Dwordd[\maskx'])s_{2k-1}$
where $k=j_{\nu_{I_f}^{-1}(i)}+1$ if $i\in\nu_{I_f}(\lshft N_{I_f})$ and $k=j_{\rshft\mu_-(i)}$ otherwise.
\item If $i\notin\nu_{I_f}(\rshft N_{I_f}\cap I_e)$ and either $\maskx_i=(0,1,\epsilon_1,\epsilon_2)$, $\maskx'_i=(0,0,\epsilon_1,\epsilon_2)$
or $\maskx_i=(1,1,0,1)$, $\maskx'_i=(1,1,0,0)$ then $\pi(\Dwordd[\maskx])=s_{2k-1}\pi(\Dwordd[\maskx'])$
where $k=j_{\nu_{I_f}^{-1}(i)}+1$ if $i\in\nu_{I_f}(\rshft N_{I_f})$ and $k=j_{\lshft\mu_-(i)}$ otherwise.
\end{enumerate}
\item Suppose that $i\in N_{I_f}$ is such that $\maskx_j=\maskx'_j$ for all $j\ne i,\nu_{I_f}(i)$.
Assume that $\maskx'_{i,r}=\maskx_{i,r}$ for $r=1,2,4$, $\maskx'_{i,3}=1-\maskx_{i,3}$,
$\maskx'_{\nu_{I_f}(i),r}=\maskx_{\nu_{I_f}(i),r}$ for $r=1,3,4$, $\maskx'_{\nu_{I_f}(i),2}=1-\maskx_{\nu_{I_f}(i),2}$
and $\maskx_{i_1,4}=\maskx'_{i_2,1}=0$ where $i_1=\min(i,\nu_{I_f}(i))$, $i_2=\max(i,\nu_{I_f}(i))$.
Then $\pi(\Dwordd[\maskx])=\pi(\Dwordd[\maskx'])$.
\end{enumerate}
\end{lemma}

\begin{proof}
Part \ref{case: triv} is trivial.
Part \ref{case: rshft} follows from the braid relation
\[
s_{2j_i}s_{2j_i+1}s_{2j_i}=s_{2j_i+1}s_{2j_i}s_{2j_i+1}
\]
and the relation
\begin{gather*}
s_{2j_i+1}\duble{s_{j_i+1}}\dots\duble{s_k}=\duble{s_{j_i+1}}\dots\duble{s_k}s_{2k+1}
\end{gather*}
where $k=j_{\nu_{I_f}(i)}-1$ if $i\in\rshft N_{I_f}$ and $k=j_{\rshft\mu_+(i)}$ otherwise.
The other parts are similar.
\end{proof}

Next, we explicate, for any $\Dwordd$-mask $\maskx\in(\{0,1\}^4)^l$, the $H$-double coset of $\pi(\Dwordd[\maskx])$, thereby finishing the proof of Proposition \ref{prop: paramdc}.

\begin{lemma} \label{lem: descoset}
For any $\maskx\in(\{0,1\}^4)^l$ let $Q_\maskx=(I_e,\rshft{I}\cup\lshft{I},I_f)$ where
\begin{subequations} \label{eq: Imaskx}
\begin{gather} \label{eq: Ief}
I_e=\{i=1,\dots,l:\maskx_i\ne(0,*,*,0),(1,0,0,1)\},\\
I_f=\{i=1,\dots,l:\maskx_i=(1,1,1,1)\},\\
\label{eq: rshftI}
\rshft{I}=\{i\in\rshft{N}_{I_f}^{I_e}:\maskx_{i,1}\cdot\maskx_{i,3}\ne\maskx_{\nu_{I_f}^{I_e}(i),2}\cdot\maskx_{\nu_{I_f}^{I_e}(i),4}\},\\
\label{eq: lshftI}
\lshft{I}=\{i\in\lshft{N}_{I_f}^{I_e}:\maskx_{i,3}\cdot\maskx_{i,4}\ne\maskx_{\nu_{I_f}^{I_e}(i),1}\cdot\maskx_{\nu_{I_f}^{I_e}(i),2}\}.
\end{gather}
\end{subequations}
Then $\pi(\Dwordd[\maskx])\in H\iprm(Q_{\maskx})H$ and hence $\prm(\pi(\Dwordd[\maskx]))=Q_{\maskx}$. In particular, the image of $\iprm$ is $\red_{\le\duble w}$.
\end{lemma}

\begin{proof}
Consider the graph $\graph_1$ whose vertex set consists of the $\Dwordd$-masks and the edges connect two $\Dwordd$-masks $\maskx,\maskx'\in(\{0,1\}^4)^l$ if
there exists $i=1,\dots,l$ such that $\maskx_j=\maskx'_j$ for all $j\ne i$ and one of the following conditions holds for some $\epsilon\in\{0,1\}$
(where $Q_\maskx=(I_e,\rshft{I}\cup\lshft{I},I_f)$):
\begin{enumerate}
\item $\maskx_i=(0,0,0,0)$ and either $\maskx'_i=(0,*,*,0)$ or $\maskx'_i=(1,0,0,1)$.
\item $\maskx_i=(1,0,0,0)$ and $\maskx'_i=(0,0,0,1)$.
\item $\maskx_i=(1,\epsilon,1,0)$, $\maskx'_i=(1,\epsilon,0,0)$ and $i\notin\rshft N_{I_f}^{I_e}$.
\item $\maskx_i=(0,\epsilon,1,1)$, $\maskx'_i=(0,\epsilon,0,1)$ and $i\notin\lshft N_{I_f}^{I_e}$.
\item $\maskx_i=(1,1,\epsilon,0)$, $\maskx'_i=(1,0,\epsilon,0)$ and $i\notin\nu_{I_f}^{I_e}(\lshft N_{I_f}^{I_e})$.
\item $\maskx_i=(0,1,\epsilon,1)$, $\maskx'_i=(0,0,\epsilon,1)$ and $i\notin\nu_{I_f}^{I_e}(\rshft N_{I_f}^{I_e})$.
\item $\maskx_i=(1,0,1,1)$ and $\maskx'_i=\begin{cases}(1,0,1,0)&i\in\rshft N_{I_f}^{I_e},\\
(0,0,1,1)&i\in\lshft N_{I_f}^{I_e},\\(0,0,0,1)&i\notin N_{I_f}^{I_e}.\end{cases}$
\item $\maskx_i=(1,1,0,1)$ and $\maskx'_i=\begin{cases}(0,1,0,1)&i\in\nu_{I_f}^{I_e}(\rshft N_{I_f}^{I_e}),\\
(1,1,0,0)&i\in\nu_{I_f}^{I_e}(\lshft N_{I_f}^{I_e}),\\(1,0,0,0)&i\notin\nu_{I_f}^{I_e}(N_{I_f}^{I_e}).\end{cases}$
\end{enumerate}
It follows from the first part of Lemma \ref{lem: sm} that the double coset $H\pi(\Dwordd[\maskx])H$ depends only on the $\graph_1$-connected component
of $\maskx$. It is also straightforward to check that $Q_{\maskx}$ depends only on the $\graph_1$-connected component of $\maskx$.

On the other hand, each $\graph_1$-connected component contains a representative $\maskx$ which satisfies the following conditions for all $i$
\begin{enumerate}
\item If $i\notin I_e$ then $\maskx_i=(0,0,0,0)$.
\item If $\maskx_i=(1,*,*,1)$ then $\maskx_i=(1,1,1,1)$, i.e., $i\in I_f$.
\item If $\maskx_i=(1,*,1,0)$ then $i\in\rshft N_{I_f}^{I_e}$.
\item \label{cond: 4} If $\maskx_i=(0,*,1,1)$ then $i\in\lshft N_{I_f}^{I_e}$.
\item If $\maskx_i=(1,1,*,0)$ then $i\in\nu_{I_f}^{I_e}(\lshft N_{I_f}^{I_e})$.
\item \label{cond: 6} If $\maskx_i=(0,1,*,1)$ then $i\in\nu_{I_f}^{I_e}(\rshft N_{I_f}^{I_e})$.
\end{enumerate}
We call such $\maskx$ ``special''. We will show that if $\maskx$ is special then $\pi(\Dwordd[\maskx])=\iprm(Q_\maskx)$,
thereby finishing the proof.

Consider a second graph $\graph_2$ with the same vertex set as $\graph_1$, where the edges are given by the condition in the second part of Lemma \ref{lem: sm}
as well as the condition that there exists $i$ such that $\maskx_j=\maskx'_j$ for all $j\ne i$ and $\maskx_i=(1,0,0,0)$, $\maskx'_i=(0,0,0,1)$.
Thus, $\pi(\Dwordd[\maskx])$ depends only on the $\graph_2$-connected component of $\maskx$
and once again, it is easy to verify that the same is true for $Q_{\maskx}$.
Note that a $\graph_2$-neighbor of a special $\Dwordd$-mask is also special.

We claim that the $\graph_2$-connected component of a special $\Dwordd$-mask $\maskx$ contains one which vanishes at all coordinates
$(i,2)$ for $i\notin I_f$.
We argue by induction on the number of indices $i\notin I_f$ such that $\maskx_{i,2}=1$.
For the induction step take such $i$ with $j_i$ minimal. Since $\maskx$ is special, by the first two conditions we have $\maskx_{i,1}+\maskx_{i,4}=1$.
Suppose for instance that $\maskx_{i,1}=0$ and $\maskx_{i,4}=1$. Then, by condition \ref{cond: 6} $i\in \nu_{I_f}^{I_e}(i_1)$ for some
$i_1\in\rshft N_{I_f}^{I_e}$. By minimality of $j_i$ we have $\maskx_{i_1,2}=0$. Also, by passing to a $\graph_2$-neighbor if necessary,
we may assume that $\maskx_{i_1}\ne(0,0,0,1)$. Then by condition \ref{cond: 4} we necessarily have $\maskx_{i_1,4}=1$ since $i_1\notin\lshft N_{I_f}^{I_e}$.
Thus, we can apply the induction hypothesis to the neighbor of $\maskx$ in $\graph_2$ which differs from it precisely
at the coordinates $(i,2)$ and $(i_1,3)$. The case $\maskx_{i,4}=0$ and $\maskx_{i,1}=1$ is similar.

Finally, if $\maskx$ is special and $\maskx_{i,2}=0$ for all $i\notin I_f$ then
$\Dwordd[\maskx]=\omegaw_{Q_{\maskx}}$ (see \eqref{def: omegaQ}) and hence $\pi(\Dwordd[\maskx])=\iprm(Q_{\maskx})$.
The lemma follows.
\end{proof}

\begin{example}
Consider the case $n=2$ and $w=s_1$ (so that $\wordd=\s_1$ and $\Dwordd=\s_2\s_1\s_3\s_2$). There are three $H$-double cosets. As representatives we can take
the identity, $s_2$ and $\duble{s_1}$. The corresponding triplets under $\prm$ are $(\emptyset,\emptyset,\emptyset)$,
$(\{1\},\emptyset,\emptyset)$ and $(\{1\},\emptyset,\{1\})$,
We have
\begin{gather*}
\pi(\wordd[\maskx])\in H\iff \maskx\in\{(0,*,*,0),(1,0,0,1)\},\\
\pi(\wordd[\maskx])\in Hs_2H\iff \maskx\in\{(1,*,*,0),(0,*,*,1),(1,0,1,1),(1,1,0,1)\},\\
\pi(\wordd[\maskx])\in H\duble{s_1}\iff \maskx=(1,1,1,1).
\end{gather*}
\end{example}

\begin{remark} \label{rem: passtoinv}
Consider the reduced decomposition $\wordd^\rev$ for $w^{-1}$.
Write $i^\rev=l+1-i$ and similarly for sets. Then for any $A\subset B\subset\{1,\dots,l\}$ we have $\,^{\wordd^\rev}\rshft N_{A^\rev}^{B^\rev}=(\,^\wordd\lshft N_A^B)^\rev$,
$\,^{\wordd^\rev}\lshft N_{A^\rev}^{B^\rev}=(\,^\wordd\rshft N_A^B)^\rev$,  $\,^{\wordd^\rev}N_{A^\rev}^{B^\rev}=(\,^\wordd N_A^B)^\rev$
and $\,^{\wordd^\rev}\nu_{A^\rev}^{B^\rev}(i)=\,^\wordd\nu_A^B(i^\rev)$.
Moreover, the following diagram is commutative
\[
\begin{CD}
\Quadr^\wordd @>{(I_e,I,I_f)\mapsto(I_e^\rev,I^\rev,I_f^\rev)}>>  \Quadr^{\wordd^\rev}\\
@VV{\iprm^\wordd}V  @VV{\iprm^{\wordd^\rev}}V\\
\red_{\le\duble w} @>{\ \ \ \ \ \ u\mapsto u^{-1}\ \ \ \ \ \ }>> \red_{\le\duble w^{-1}}
\end{CD}
\]
\end{remark}

\section{The main result}
Finally, we prove the main result of the paper.
Recall that $w\in S_n$ is a fixed Boolean permutation with reduced decomposition $\wordd=\s_{j_1}\dots\s_{j_l}$ (with $j_1,\dots,j_l$ distinct).
\begin{proposition} \label{prop: mainprop}
For any $Q=(I_e,I,I_f)\in\Quadr$ we have
\begin{multline} \label{eq: altsum}
\sum_{u\in H\iprm(Q)H}\sgn u\, P_u^{\Dwordd}=
\sum_{\maskx\in(\{0,1\}^4)^l:\prm(\pi(\Dwordd[\maskx]))=Q}\sgn\maskx\, q^{\dfct_{\Dwordd}(\maskx)}\\=
(-1)^{\card{I_e\setminus I_f}}
q^{\card{I_e^c}+\card{N_{I_f}^{I_e}\setminus I}}(q+1)^{\card{I_e\setminus (I_f\cup N_{I_f}^{I_e})}}.
\end{multline}
In particular, 
\begin{equation} \label{eq: altsum0}
\sum_{\maskx\in(\{0,1\}^4)^l:\prm(\pi(\Dwordd[\maskx]))=Q}\sgn\maskx=
2^{\card{I_e\setminus (I_f\cup N_{I_f}^{I_e})}}(-1)^{\card{I_e\setminus I_f}}.
\end{equation}
\end{proposition}

By Theorem \ref{thm: Deodhar} and Corollary \ref{cor: dtight} we infer
\begin{corollary} \label{cor: maincor}
For any $Q=(I_e,I,I_f)\in\Quadr$ we have
\[
\sum_{u\in H\iprm(Q)H}\sgn u\, P_{u,\duble w}=(-1)^{\card{I_e\setminus I_f}}q^{\card{I_e^c}+\card{N_{I_f}^{I_e}\setminus I}}
(q+1)^{\card{I_e\setminus (I_f\cup N_{I_f}^{I_e})}}.
\]
In particular, by \eqref{eq: qtau}, for any $x\le w$
\[
\tilde P^{(2)}_{x,w}=\sum_{u\in H\duble x}\sgn u\, P_{u,\duble w}=q^{\ell(w)-\ell(x)}.
\]
\end{corollary}

\begin{proof}
The first equality of \eqref{eq: altsum} follows from Proposition \ref{prop: paramdc}.
We prove the second one by induction on $l$.
The case $l=0$ is trivial -- both sides of \eqref{eq: altsum} are equal to $1$.
Suppose that $l>0$ and the result is known for $l-1$.

If $I_f=\{1,\dots,l\}$ (so that $Q=(\{1,\dots,l\},\emptyset,\emptyset,\{1,\dots,l\})$) then the only summand
on the left-hand side of \eqref{eq: altsum} is the one corresponding to $\maskx_i=(1,1,1,1)$ for all $i$
and the result is trivial.

We may therefore assume that $I_f\ne\{1,\dots,l\}$.
For convenience, denote the left-hand side of \eqref{eq: altsum} by $L^\wordd_Q$ and let
\[
M_Q^\wordd=\{\maskx\in(\{0,1\}^4)^l:\prm(\pi(\Dwordd[\maskx]))=Q\}
\]
which is explicated in Lemma \ref{lem: descoset}.
Let $i_0$ be the element of $I_f^c$ for which $j_{i_0}$ is maximal.
In particular, $\maskx_{i_0}\ne(1,1,1,1)$ for any $\maskx\in M_Q^\wordd$.
Clearly $i_0\notin N_{I_f}$, otherwise $j_{\nu_{I_f}(i_0)}>i_0$.
Note that by  \eqref{eq: Pinv} and Remark \ref{rem: passtoinv}, the statement of Corollary \ref{cor: maincor} is invariant under $w\mapsto w^{-1}$
(and $\wordd\mapsto\wordd^\rev$).
On the other hand, Corollary \ref{cor: maincor} is equivalent to Proposition \ref{prop: mainprop} by Theorem \ref{thm: Deodhar}.
Therefore, upon inverting $w$ if necessary we may assume that
\begin{equation} \label{eq: notonleft}
i_0\notin\nu_{I_f}(\lshft N_{I_f}).
\end{equation}
In particular, we can apply Corollary \ref{cor: changed}.

Let $\wordd'$ be the word obtained from $\wordd$ by removing $\s_{j_{i_0}}$ and let 
$Q'=(I_e',I',I_f')\in\Quadr^{\wordd'}$ where $I'_e=I_e\setminus\{i_0\}$, $I'=I\setminus(\nu_{I_f}^{I_e})^{-1}(\{i_0\})$ and $I_f'=I_f$.
Note that $\,^{\wordd'}N_{I_f'}^{I_e'}=N_{I_f}^{I_e}\setminus(\nu_{I_f}^{I_e})^{-1}(\{i_0\})$.
(For simplicity we suppress $\wordd$ if the notation is pertaining to it.)

To carry out the induction step we show using Lemmas \ref{lem: detdfct} and \ref{lem: descoset} that
\begin{equation} \label{eq: recLwQ}
L^\wordd_Q=L^{\wordd'}_{Q'}\times\begin{cases}q&i_0\notin I_e,\\
-(q+1)&i_0\in I_e\setminus\nu_{I_f}^{I_e}(N_{I_f}^{I_e}),\\
-1&i_0\in\nu_{I_f}^{I_e}(I),\\
-q&i_0\in\nu_{I_f}^{I_e}(N_{I_f}^{I_e}\setminus I).\end{cases}
\end{equation}

We separate into cases.
\begin{enumerate}
\item Assume first that $i_0\notin\nu_{I_f}^{I_e}(N_{I_f}^{I_e})$ (the first two cases on the right-hand side of \eqref{eq: recLwQ}).
In this case, in order for $\maskx$ to belong to $M_Q^\wordd$, the conditions \eqref{eq: rshftI} and \eqref{eq: lshftI} are independent of $\maskx_{i_0}$.

We claim that
\begin{equation} \label{eq: onlyspec}
L^\wordd_Q=\sum_{\maskx\in M_Q^\wordd:\maskx_{i_0}\in \{(1,0,0,0),(1,0,0,1),(1,1,1,0)\}}\sgn\maskx\, q^{\dfct_{\Dwordd}(\maskx)},
\end{equation}
i.e., that
\begin{equation} \label{eq: sum0}
\sum_{\maskx\in R_{i_0}}\sgn\maskx\, q^{\dfct_{\Dwordd}(\maskx)}=0
\end{equation}
where
\[
R_{i_0}:=\{\maskx\in M_Q^\wordd:\maskx_{i_0}\ne (1,0,0,*),(1,1,1,*)\}.
\]
We define an involution $\iota$ on $R_{i_0}$ by retaining $\maskx_i$ for $i\ne i_0$ and changing $\maskx_{i_0}$
according to the rule
\begin{gather*}
(0,0,0,0)\leftrightarrow (0,0,1,0),\ \ (0,1,0,0)\leftrightarrow(0,1,1,0),\ \ (1,0,1,0)\leftrightarrow (1,0,1,1)
,\\(0,0,0,1)\leftrightarrow (0,0,1,1),\ \ (0,1,0,1)\leftrightarrow(0,1,1,1),\ \
(1,1,0,0)\leftrightarrow(1,1,0,1).
\end{gather*}
This is well defined since the condition $\maskx_{i_0}=(0,*,*,0)$ is invariant under the above rule.
By the first part of Corollary \ref{cor: changed} $\iota$ preserves $\dfct_{\Dwordd}$.
Since $\sgn\iota(\maskx)=-\sgn\maskx$, the assertion \eqref{eq: sum0} follows.

\begin{enumerate}
\item Suppose that $i_0\notin I_e$.
Then by \eqref{eq: onlyspec}
\[
L^\wordd_Q=\sum_{\maskx\in M_Q^\wordd:\maskx_{i_0}=(1,0,0,1)}\sgn\maskx\, q^{\dfct_{\Dwordd}(\maskx)}
\]
and therefore by the second part of Corollary \ref{cor: changed}
\[
L^\wordd_Q=qL^{\wordd'}_{Q'}.
\]
\item Similarly, if $i_0\in I_e\setminus\nu_{I_f}^{I_e}(N_{I_f}^{I_e})$ then
\[
L^\wordd_Q=\sum_{\maskx\in M_Q^\wordd:\maskx_{i_0}\in \{(1,0,0,0),(1,1,1,0)\}}\sgn\maskx\, q^{\dfct_{\Dwordd}(\maskx)}
\]
and we get
\[
L^\wordd_Q=-(q+1)L^{\wordd'}_{Q'}.
\]
\end{enumerate}

\item Consider now the case $i_0\in\nu_{I_f}^{I_e}(N_{I_f}^{I_e})$ (the last two cases on the right-hand side of \eqref{eq: recLwQ}).
In particular, $i_0\in I_e\setminus I_f$ so that
\[
\maskx_{i_0}\in\{(1,*,*,0),(0,*,*,1),(1,0,1,1),(1,1,0,1)\}
\]
for any $\maskx\in M_Q^\wordd$.

Let $i_1=(\nu_{I_f}^{I_e})^{-1}(i_0)$. By \eqref{eq: notonleft} $i_1\in\rshft{N}_{I_f}^{I_e}$.
We write $L^\wordd_Q=T_0+T_1$ where
\[
T_j=\sum_{\maskx\in M_Q^\wordd:\maskx_{i_1,1}\cdot\maskx_{i_1,3}=j}\sgn\maskx\, q^{\dfct_{\Dwordd}(\maskx)},\ \ j=0,1.
\]
We first claim that
\begin{equation} \label{eq: T_jrest}
T_j=\sum_{\substack{\maskx\in M_Q^\wordd:\maskx_{i_1,1}\cdot\maskx{i_1,3}=j\text{ and }\\
\maskx_{i_0}\in\{(1,0,0,0),(1,1,1,0),(1,1,0,0),(1,1,0,1)\}}}\sgn\maskx\, q^{\dfct_{\Dwordd}(\maskx)},
\end{equation}
or in other words, that
\[
\sum_{\maskx\in R_j}\sgn\maskx\, q^{\dfct_{\Dwordd}(\maskx)}=0
\]
where
\[
R_j:=\{\maskx\in M_Q^\wordd:\maskx_{i_1,1}\cdot\maskx_{i_1,3}=j\text{ and }
\maskx_{i_0}\notin\{(1,0,0,0),(1,1,1,0),(1,1,0,0),(1,1,0,1)\}\}.
\]
As before, we define $\iota$ on $R_j$ by keeping $\maskx_i$ for all $i\ne i_0$ and changing $\maskx_{i_0}$ according to the rule
\[
(0,0,0,1)\leftrightarrow(0,0,1,1),\ (0,1,0,1)\leftrightarrow(0,1,1,1),\ (1,0,1,0)\leftrightarrow(1,0,1,1).
\]
This is well defined since $\iota$ preserves $\maskx_{i_0,2}\cdot\maskx_{i_0,4}$.
By the first part of Corollary \ref{cor: changed} $\iota$ preserves $\dfct_{\Dwordd}$. Since
$\sgn\iota(\maskx)=-\sgn\maskx$, the assertion follows.

\begin{enumerate}
\item Suppose that $i_1\notin I$. Then by \eqref{eq: T_jrest} and \eqref{eq: rshftI}
\[
T_0=\sum_{\substack{\maskx\in M_Q^\wordd:\maskx_{i_1,1}\cdot\maskx_{i_1,3}=0\text{ and }\\
\maskx_{i_0}\in\{(1,0,0,0),(1,1,1,0),(1,1,0,0)\}}}\sgn\maskx\, q^{\dfct_{\Dwordd}(\maskx)}
\]
and
\[
T_1=\sum_{\maskx\in M_Q^\wordd:\maskx_{i_1,1}\cdot\maskx_{i_1,3}=1\text{ and }\maskx_{i_0}=(1,1,0,1)}\sgn\maskx\, q^{\dfct_{\Dwordd}(\maskx)}.
\]
Thus, using the second part of Corollary \ref{cor: changed} the contributions from $\maskx_{i_0}=(1,1,1,0)$ and $\maskx_{i_0}=(1,1,0,0)$ cancel
and we get
\[
T_j=(-q)\times\sum_{\maskx'\in M_{Q'}^{\wordd'}:\maskx'_{i_1,1}\cdot\maskx'_{i_1,3}=j}\sgn\maskx'\, q^{\dfct_{\Dwordd'}(\maskx')},\ \ j=0,1.
\]
Hence,
\[
T_0+T_1=(-q)\times\sum_{\maskx'\in M_{Q'}^{\wordd'}}\sgn\maskx'\, q^{\dfct_{\Dwordd'}(\maskx')}.
\]
\item Similarly, if $i_1\in I$ (hence $i_1\in\rshft I$) then
\[
T_0=\sum_{\maskx\in M_Q^\wordd:\maskx_{i_1,1}\cdot\maskx_{i_1,3}=0\text{ and }\maskx_{i_0}=(1,1,0,1)}\sgn\maskx\, q^{\dfct_{\Dwordd}(\maskx)},
\]
and
\[
T_1=\sum_{\substack{\maskx\in M_Q^\wordd:\maskx_{i_1,1}\cdot\maskx_{i_1,3}=1\text{ and }\\\maskx_{i_0}\in\{(1,0,0,0),(1,1,1,0),(1,1,0,0)\}}}
\sgn\maskx\, q^{\dfct_{\Dwordd}(\maskx)}.
\]
The contributions from $\maskx_{i_0}=(1,0,0,0)$ and $\maskx_{i_0}=(1,1,0,0)$ cancel and we obtain
\[
T_j=-\sum_{\maskx'\in M_{Q'}^{\wordd'}:\maskx'_{i_1,1}\cdot\maskx'_{i_1,3}=j}\sgn\maskx'\, q^{\dfct_{\Dwordd'}(\maskx')},\ \ j=0,1
\]
and hence
\[
T_0+T_1=-\sum_{\maskx'\in M_{Q'}^{\wordd'}}\sgn\maskx'\, q^{\dfct_{\Dwordd'}(\maskx')}.
\]
\end{enumerate}
\end{enumerate}
Thus, we established \eqref{eq: recLwQ} in all cases.
The induction step now follows from the induction hypothesis.
\end{proof}

\section{Complements}

In conclusion, we relate the result of the previous section to the results of \cite[\S10]{1605.08545}.
We continue to assume that $w$ and $\wordd$ are as in \S\ref{sec: ddc}.

\begin{lemma} \label{lem: ncyc}
Let $x_1,x_2\le w$, $I_e=I_{x_1}\cup I_{x_2}$ and $I_f=I_{x_1}\cap I_{x_2}$.
Then the number of non-trivial cycles of the permutation $x_2^{-1}x_1$ is $\card{I_e\setminus(I_f\cup N_{I_f}^{I_e})}$.
\end{lemma}

\begin{proof}
We prove this by induction on the cardinality of $I_f$. If $I_f$ is the empty set then $x_2^{-1}x_1$ is a Boolean permutation,
$\{j:s_j\le x_2^{-1}x_1\}=\{j_i:i\in I_e\}$ and
\[
N_{I_f}^{I_e}=\{i\in I_e:\text{there exists $r\in I_e$ such that }j_i+1=j_r\}.
\]
The claim follows since any Coxeter element of the symmetric group is a single cycle.

For the induction step suppose that $I_f\ne\emptyset$. Let $i_1$ be the index in $I_f$ with $j_{i_1}$ minimal.

\begin{enumerate}
\item Suppose first that $j_{i_1}-1\ne j_{i'}$ for all $i'\in I_e$.
If $j_{i_1}+1=j_{i_2}$ for some $i_2\in I_e$ then we may assume upon replacing $w$ by $w^{-1}$ (and $\wordd$ by $\wordd^\rev$) if necessary
that $i_2>i_1$. Let $x'_1,x'_2\le w$ be such that $I_{x'_r}=I_{x_r}\setminus\{i_1\}$.
Then $x_2^{-1}x_1=(x_2')^{-1}x'_1$.
Letting $I'_e=I_{x'_1}\cup I_{x'_2}=I_e\setminus\{i_1\}$ and $I'_f=I_{x'_1}\cap I_{x'_2}=I_f\setminus\{i_1\}$
we have $N_{I'_f}^{I'_e}=N_{I_f}^{I_e}$ and therefore $I'_e\setminus(I'_f\cup N_{I'_f}^{I'_e})=I_e\setminus(I_f\cup N_{I_f}^{I_e})$.
Thus, the claim follows from the induction hypothesis.

\item Otherwise, $j_{i_1}-1=j_{i_0}$ for some $i_0\in I_e\setminus I_f$ (by the minimality of $i_1$).
Once again, upon replacing $w$ by $w^{-1}$ (and $\wordd$ by $\wordd^\rev$) if necessary we may assume that $i_1<i_0$.
\begin{enumerate}
\item Suppose first that $i_0\notin N_{I_f}^{I_e}$.
Let $t\ge1$ be the maximal index for which there exist indices $i_t<\dots<i_1$ in $I_f$ such that $j_{i_r}=j_{i_0}+r$, $r=1,\dots,t$.
By the assumption on $i_0$ and $t$, if there exists $i'\in I_e$ such that $j_{i'}=j_{i_t}+1$ then $i'>i_t$. Therefore
$x_2^{-1}x_1=(x'_2)^{-1}x'_1$ where $x_1',x_2'\le w$ are such that $I_{x'_r}=I_{x_r}\setminus\{i_1,\dots,i_t\}$.
Let $I'_e=I_{x_1'}\cup I_{x_2'}=I_e\setminus\{i_1,\dots,i_t\}$ and $I'_f=I_{x'_1}\cap I_{x'_2}=I_f\setminus\{i_1,\dots,i_t\}$.
Then $N_{I'_f}^{I'_e}=N_{I_f}^{I_e}$. This case therefore follows from the induction hypothesis.
\item Suppose that $i_0\in N_{I_f}^{I_e}$ and let
$i_{t+1}<\dots<i_1$, $t\ge1$ be such that $j_{i_r}=j_{i_0}+r$, $r=1,\dots,t+1$ with $i_1,\dots,i_t\in I_f$ and $i_{t+1}=\nu_{I_f}^{I_e}(i_0)\in I_e\setminus I_f$.
Upon interchanging $x_1$ and $x_2$ if necessary we may assume that $i_0\in I_{x_1}\setminus I_{x_2}$.
Let $u$ be the permutation
\[
u(r)=\begin{cases}r+j_{i_0}&r\le t,\\r-t&t<r\le j_{i_t},\\r&r>j_{i_t}\end{cases}
\]
so that
\[
u^{-1}s_{j_{i_1}}s_{j_{i_2}}\dots s_{j_{i_{t+1}}}s_{j_{i_t}}\dots s_{j_{i_0}}u=s_{j_{i_{t+1}}}s_{j_{i_t}}\text{ and }
u^{-1}s_ru=s_{r+t}\text{ for all }r<j_{i_0}.
\]
Let $\wordd'$ be the word obtained from $\wordd$ by removing the $t$ simple roots $\s_{j_{i_1}},\dots,\s_{j_{i_r}}$ (i.e., the indices
$i_1,\dots,i_t$) and replacing $\s_r$ by $\s_{r+t}$ for $r\le j_{i_0}$.
Then $u^{-1}x_2^{-1}x_1u=(x_2')^{-1}x_1'$ where $I_{x_r'}^{\wordd'}=I_{x_r}\setminus\{i_1,\dots,i_t\}$, $r=1,2$.
Let $I'_e=I_{x_1'}\cup I_{x_2'}=I_e\setminus\{i_1,\dots,i_t\}$ and $I'_f=I_{x'_1}\cap I_{x'_2}=I_f\setminus\{i_1,\dots,i_t\}$.
Then $\,^{\wordd'}N_{I'_f}^{I'_e}=N_{I_f}^{I_e}$ and the claim follows from the induction hypothesis.
\end{enumerate}
\end{enumerate}
\end{proof}

Let $K$ be the subgroup of $S_{2n}$ (isomorphic to $S_n\times S_n$) preserving the set $\{2,4,\dots,2n\}$.
For any $x\in K$ let $x_{\odd}\in S_n$ (resp., $x_{\even}\in S_n$) be the permutation such that
$x(2i-1)=2x_{\odd}(i)-1$ (resp., $x(2i)=2x_{\even}(i)$) for $i=1,\dots,n$.
Thus, $x\mapsto(x_{\odd},x_{\even})$ is a group isomorphism $K\simeq S_n\times S_n$.
Note that $x_{\even}=x_{\odd}$ if and only if $x=\duble{x_{\odd}}$.

We recall the following general elementary result.
\begin{lemma} \label{lem: LM} \cite[Lemma 10.6]{1605.08545}
For any $x\in S_{2n}$ we have $K\cap HxH\ne\emptyset$. Let $u\in K\cap HxH$ and let $r$ be the number of non-trivial cosets
of $u_{\even}^{-1}u_{\odd}\in S_n$. Then $\card{K\cap HxH}=2^r$. Moreover, $\sgn$ is constant on $K\cap HxH$.
\end{lemma}

Note that in general, for any $w\in S_n$, if $u\in K$ and $u_{\odd},u_{\even}\le w$ then $u\le\duble w$.
In the case at hand we can be more explicit.
\begin{lemma} \label{lem: dbls1s2}
Let $u\in K$ with $u_{\odd},u_{\even}\le w$. Then
\begin{equation} \label{eq: iotamask}
u=\pi(\Dwordd[\maskx])\ \ \text{ where }\maskx_i=(1,\chi_{I_{u_{\odd}}}(i),\chi_{I_{u_{\even}}}(i),1),\ \ i=1,\dots,l
\end{equation}
and $\prm(u)=(I_e,I,I_f)$ where $I_e=I_{u_{\odd}}\cup I_{u_{\even}}$, $I_f=I_{u_{\odd}}\cap I_{u_{\even}}$ and
\begin{multline} \label{eq: getI1I2}
I=N_{I_f}^{I_e}\setminus(N_{I_f}^{I_{u_{\odd}}}\cup N_{I_f}^{I_{u_{\even}}})\\=
\{i\in N_{I_f}\cap I_{u_{\odd}}:\nu_{I_f}(i)\in I_{u_{\even}}\}\cup\{i\in N_{I_f}\cap I_{u_{\even}}:\nu_{I_f}(i)\in I_{u_{\odd}}\}.
\end{multline}
\end{lemma}

\begin{proof}
Since $u\mapsto(u_{\odd},u_{\even})$ is a group isomorphism, it is enough to check \eqref{eq: iotamask} for the case $n=1$, which is straightforward.
The second part follows from Lemma \ref{lem: descoset}.
\end{proof}

\begin{remark} \label{rem: equ}
Let $A$ and $B$ be subsets of $\{1,\dots,l\}$ with $A\subset B$.
Let $\sim$ be the equivalence relation on $B\setminus A$ generated by $i\sim\nu_A^B(i)$ whenever $i\in N_A^B$.
Any equivalence class $C\subset B\setminus A$ of $\sim$ is of the form
\begin{equation} \label{def: C}
C=\{i_1,\dots,i_a\}
\end{equation}
where
\begin{enumerate}
\item For all $t<a$, $i_t\in N_A^B$ and $\nu_A^B(i_t)=i_{t+1}$. In particular, $j_{i_{t+1}}>j_{i_t}$.
\item $i_a\notin N_A^B$.
\item $i_1\notin\nu_A^B(N_A^B)$.
\end{enumerate}
Thus, each equivalence class contains a unique element outside $N_A^B$.
In particular, the number of equivalence classes of $\sim$ is $\card{B\setminus(A\cup N_A^B)}$.
\end{remark}

\begin{corollary} \label{cor: niota}
For any $Q=(I_e,I,I_f)\in\Quadr$
\[
\card{\{u\in K:u_{\even},u_{\odd}\le w\text{ and }\prm(u)=Q\}}=2^{\card{I_e\setminus(I_f\cup N_{I_f}^{I_e})}}.
\]
Moreover,
\[
\sgn u=(-1)^{\card{I_e\setminus I_f}}
\]
for any $u\in K$ such that $u_{\even},u_{\odd}\le w$ and $\prm(u)=Q$.
\end{corollary}

\begin{proof}
Indeed, by \eqref{eq: booleanmap} and Lemma \ref{lem: dbls1s2}, the set on the left-hand side is in bijection with the set of ordered pairs $(I_1,I_2)$ of subsets of $\{1,\dots,l\}$
such that $I_1\cap I_2=I_f$, $I_1\cup I_2=I_e$ and \eqref{eq: getI1I2} holds. Under this bijection $\sgn u=
(-1)^{\card{I_1\triangle I_2}}$ and the symmetric difference $I_1\triangle I_2$ is equal to $I_e\setminus I_f$. This implies the second part.
Now, the map $(I_1,I_2)\mapsto I_1\setminus I_2$ is a bijection between
\[
\{(I_1,I_2):I_1,I_2\subset\{1,\dots,l\},I_1\cap I_2=I_f\text{ and }I_1\cup I_2=I_e\}
\]
and $\power(I_e\setminus I_f)$. 
Moreover, the condition \eqref{eq: getI1I2} holds if and only if for every equivalence class \eqref{def: C} of $\sim$ as above with respect to $A=I_f$ and $B=I_e$
and every $t<a$ we have $\chi_{I_1\setminus I_2}(i_{t+1})=\chi_{I_1\setminus I_2}(i_t)$ if and only $i_t\notin I$.
Thus,  by Remark \ref{rem: equ} $\chi_{I_1\setminus I_2}$ is determined by its values on $B\setminus(A\cup N_A^B)$, which are arbitrary.
The corollary follows.
\end{proof}

Combining the results of this section we obtain
\begin{corollary}
Let $u\in K$. Then $u\le\duble w$ if and only if $u_{\even},u_{\odd}\le w$.
Moreover, the right-hand side of \eqref{eq: altsum0} is
\[
\sum_{u\in K:u\le\duble w\text{ and }\prm(u)=Q}\sgn u=
\pm\card{H\iprm(Q)H\cap K}.
\]
\end{corollary}

We remark that the first part of the corollary holds in fact for any smooth $w$ \cite[Corollary 10.8]{1605.08545}.

\appendix
\section{Numerical results}
\subsection{}
We have calculated all the polynomials $\tilde P_{x,w}^{(m)}$, $x,w\in S_n$ and verified Conjecture \ref{conj: main}
for $nm\le 12$.
\footnote{We remark that already for $m=2$ we may have
$\deg P_{\duble x,\duble w}>\ell(w)-\ell(x)$ even if $P_{x,w}=1$, e.g. for $(w,x)=(35421,13254)$.}
(Recall that Conjecture \ref{conj: main} is known for $n=2$.)
Let us call a pair $(w,x)$ in $S_n$ \emph{reduced} if it admits no cancelable indices (see
Remark \ref{rem: basicfacts}\eqref{rem: cancelable})
and $xs<x$ (resp., $sx<x$) for any simple reflection $s$ such that $ws<w$ (resp., $sw<w$).
In the following tables we list $\tilde P_{x,w}^{(m)}$ in the cases $nm\le12$ ($n,m>1$) for all reduced pairs $(w,x)$ in $S_n$.
By Conjecture \ref{conj: main} (which we checked at the cases at hand)
and Remark \ref{rem: basicfacts}\eqref{rem: cancelable}, this covers all the polynomials $\tilde P_{x,w}^{(m)}$ without restriction on $(w,x)$.
To avoid repetitions, we only list representatives for the equivalence classes of the relation
$(w,x)\sim(w^{-1},x^{-1})\sim(w_0ww_0,w_0xw_0)\sim(w_0w^{-1}w_0,w_0x^{-1}w_0)$.

\begin{center}
\captionof{table}{$n=4$}
\[
\begin{array}{c|c|c|c}
(w,x) & P_{x,w} & \tilde P_{x,w}^{(2)} & \tilde P_{x,w}^{(3)}\\
\hline
(3412,1324),\
(4231,2143) & 1+q & 1+q+q^2 & 1+q+q^2+q^3\\
\end{array}
\]
\end{center}

\begin{center}
\captionof{table}{$n=5$}
\[
\begin{array}{c|c|c}
(w,x) & P_{x,w} & \tilde P_{x,w}^{(2)}\\
\hline
(35142,13254),\
(52431,21543) & 1+q & 1+q+q^2 \\
(34512,13425),\
(45231,24153) & 1+2q & 1+2q+3q^2 \\
(45312,14325) & 1+q^2 & 1+q^2+q^4 \\
(52341,21354) & 1+2q+q^2 & 1+2q+4q^2+2q^3+q^4 \\
\end{array}
\]
\end{center}

Note that in the cases $n=4,5$ we have $\tilde P_{x,w}^{(2)}=(P_{x,w}^2+P_{x,w}(q^2))/2$.
We split the case $n=6$ according to two subcases.

\begin{center}
\captionof{table}{Cases for $n=6$ where $\tilde P_{x,w}^{(2)}=(P_{x,w}^2+P_{x,w}(q^2))/2$}
\[
\begin{array}{c|c|c}
(w,x) & P_{x,w} & \tilde P_{x,w}^{(2)}\\
\hline
(361452,143265),\
(361542,132654),\
(426153,214365),&&\\
(562341,254163),\
(625431,216543) & 1+q & 1+q+q^2 \\
\hline
(346152,134265),\
(356142,135264),\
(462513,241635),&&\\
(462531,241653),\
(562431,251643)& 1+2q & 1+2q+3q^2 \\
\hline
(356412,135426),\
(463152,143265),\
(465132,143265),&&\\
(564231,254163),\
(632541,321654),\
(653421,321654) & 1+q+q^2 & 1+q+2q^2+q^3+q^4 \\
\hline
(456312,145326) & 1+2q^2 & 1+2q^2+3q^4 \\
\hline
(351624,132546),\
(354612,132546),\
(361452,132465),&&\\
(364152,143265),\
(456132,143265),\
(462351,243165),&&\\
(463512,143625),\
(465231,243165),\
(562341,321654),&&\\
(563421,321654),\
(623541,213654),\
(624351,214365),&&\\
(624531,214653),\
(634521,321654),\
(635241,326154) & 1+2q+q^2 & 1+2q+4q^2+2q^3+q^4 \\
\hline
(456123,145236) & 1+4q+q^2 & 1+4q+11q^2+4q^3+q^4 \\
(564312,154326) & 1+q^3 & 1+q^3+q^6\\
\end{array}
\]
\end{center}

\begin{center}
\captionof{table}{Cases for $n=6$ where $\tilde P_{x,w}^{(2)}\ne(P_{x,w}^2+P_{x,w}(q^2))/2$}
\[
\begin{array}{c|c|c}
(w,x) & P_{x,w} & \tilde P_{x,w}^{(2)} \\
\hline
(345612,134526),\
(456231,245163),&&\\
(563412,351624) & 1+3q & 1+3q+7q^2 \\
\hline
(364512,132654),\
(426351,214365),&&\\
(456123,214365),\
(623451,214365),&&\\
(563412,154326),\
(645231,216543) & 1+2q+q^2 & 1+2q+3q^2+2q^3+q^4 \\
\hline
(345612,132546),\
(356124,135246),&&\\
(364512,143625),\
(456123,143265),&&\\
(456231,215463),\
(562341,251463),&&\\
(563412,153624),\
(563412,321654),&&\\
(645231,426153) & 1+3q+q^2 & 1+3q+8q^2+3q^3+q^4 \\
\hline
(364512,132645),\
(456123,124365),&&\\
(456231,214365),\
(563412,153426),&&\\
(563412,132654),\
(645231,216453) & 1+3q+2q^2 & 1+3q+9q^2+8q^3+3q^4 \\
\hline
(462351,241365),\
(562341,231654) & 1+3q+2q^2 & 1+3q+9q^2+7q^3+3q^4 \\
\hline
(456123,125436) & 1+4q+2q^2 & 1+4q+14q^2+13q^3+3q^4 \\
(562341,251364) & 1+4q+3q^2 & 1+4q+17q^2+16q^3+6q^4 \\
(562341,231564) & 1+4q+4q^2 & 1+4q+21q^2+30q^3+14q^4 \\
(623451,213465) & 1+3q+3q^2+q^3 & 1+3q+10q^2+13q^3+10q^4+3q^5+q^6\\
(563412,132546),\
(645231,214365) & 1+3q+3q^2+q^3 & 1+3q+10q^2+16q^3+13q^4+5q^5+q^6\\
(456123,124356) & 1+4q+4q^2+q^3 & 1+4q+18q^2+34q^3+27q^4+4q^5+q^6 \\
\end{array}
\]
\end{center}

\subsection{Implementation}
For the computation, we actually wrote and executed a C program to calculate all ordinary Kazhdan--Lusztig polynomials
$P_{x,w}$ for the symmetric groups $S_k$, $k\le12$. As far as we know this computation is already new for $k=11$.
(See \cite{MR1959749} and \cite{MR2859901} for accounts of earlier computations, as well as the documentation
of the Atlas software and other mathematical software packages.)
As always, the computation proceeds with the original recursive formula of Kazhdan--Lusztig \cite{MR560412}
\[
P_{x,w}=q^cP_{x,ws}+q^{1-c}P_{xs,ws}-\sum_{z:zs<z<ws}\mu(z,ws)q^{\frac{\ell(w)-\ell(z)}2}P_{x,z}
\]
where $\mu(x,y)$ is the coefficient of $q^{(\ell(y)-\ell(x)-1)/2}$ (the largest possible degree) in $P_{x,y}$,
$s$ is a simple reflection such that $ws<w$ and $c$ is $1$ if $xs<x$ and $0$ otherwise.
However, some special features of the symmetric group allow for a faster, if ad hoc, code.
(See Remark \ref{rem: basicfacts}\eqref{rem: cancelable} and the comments below.)
For $S_{11}$ the program runs on a standard laptop (a Lenovo T470s 2017 model with 2.7 GHz CPU and 16GB RAM)
in a little less than $3$ hours.
For $S_{12}$ the memory requirements are about 500 GB RAM.
We ran it on the computer of the Faculty of Mathematics and Computer Science of the
Weizmann Institute of Science (SGI, model UV-10), which has one terabyte RAM and 2.67 GHz CPU.
The job was completed after almost a month of CPU time on a single core.

Let us give a few more details about the implementation.
We say that a pair $(w,x)$ is fully reduced if it is reduced (see above) and $x\le ws,sw$ for any simple reflection $s$.
Recall that we only need to compute $P_{x,w}$ for fully reduced pairs. The number of fully reduced pairs for $S_{12}$, up to symmetry,
is about $46\times 10^9$. However, a posteriori, the number of distinct polynomials obtained is ``only'' about
$4.3\times 10^9$. This phenomenon (which had been previously observed for smaller symmetric groups) is crucial for the implementation
since it makes the memory requirements feasible.
An equally important feature, which once again had been noticed before for smaller symmetric groups, is that
only for a small fraction of the pairs above, namely about $66.5\times 10^6$, we have $\mu(x,w)>0$.
This fact cuts down significantly the number of summands in the recursive formula and makes the computation feasible
in terms of time complexity.

We store the results as follows.
\begin{enumerate}
\item A ``glossary'' of the $\sim 4.3\times 10^9$ different polynomials. (The coefficients of the vast majority of the polynomials
are smaller than $2^{16}=65536$. The average degree is about $10$.)
\item A table with $\sim 46\times 10^9$ entries that provides for each reduced pair the pointer to $P_{x,w}$
in the glossary.
\item An additional lookup table of size $12!\sim 0.5\times 10^9$ (which is negligible compared to the previous one)
so that in the previous table we only need to record $x$ and the pointer to $P_{x,w}$
(which can be encoded in $29$ and $33$ bits, respectively), but not $w$.
\item A table with $\sim 66.5\times 10^6$ entries recording $x$, $w$, $\mu(x,w)$ for all fully reduced pairs (up to symmetry)
with $\mu(x,w)>0$.
\end{enumerate}
Thus, the main table is of size $\sim 8\times 46\times 10^9$ bytes, or about $340$ GB.
This is supplemented by the glossary table which is of size $< 100$ GB, plus auxiliary tables of insignificant size.
Of course, by the nature of the recursive algorithm all these tables have to be stored in the RAM.

We mention a few additional technical aspects about the program.
\begin{enumerate}
\item The outer loop is over all permutation $w\in S_n$ in lexicographic order.
Given $w\in S_{12}$ it is possible to enumerate efficiently the pairs $(w,x)$
such that $xs<x$ (resp., $sx<x$) whenever $ws<w$ (resp., $sw<w$). More precisely, given such $x<w$
we can very quickly find the next such $x$ in lexicographic order. Moreover, one can incorporate
the ``non-cancelability'' condition to this ``advancing'' procedure and then test the condition
$x\le ws, sw$ for the remaining $x$'s.
Thus, it is perfectly feasible to enumerate the $\sim 46\times 10^9$ fully reduced pairs.
\item On the surface, the recursive formula requires a large number of additions and multiplications in each step.
However, in reality, the number of summands is usually relatively small, since the $\mu$-function is rarely non-zero.
\item For each $w\ne 1$ the program picks the first simple root $s$ (in the standard ordering) such that $ws<w$
and produces the list of $z$'s such that $zs<z<ws$ and $\mu(z,ws)>0$. The maximal size of this list turns out to be $\sim 100,000$
but it is usually much much smaller. The list is then used to compute $P_{x,w}$ (and in particular, $\mu(x,w)$)
for all fully reduced pairs using the recursion formula and the polynomials already generated for $w'<w$.
Of course, for any given $x$ only the $z$'s with $x\le z$ matter.
\item Since we only keep the data for fully reduced pairs (in order to save memory) we need to find,
for any given pair the fully reduced pair which ``represents'' it. Fortunately, this procedure
is reasonably quick.
\item The glossary table is continuously updated and stored as 1,000 binary search trees, eventually consisting of
$\sim 4.3\times 10^6$ internal nodes each.
The data is sufficiently random so that there is no need to balance the trees. The memory overhead for maintaining the trees
is inconsequential.
\item In principle, it should be possible to parallelize the program so that it runs simultaneously on many processors.
The point is that the recursive formula only requires the knowledge of $P_{x',w'}$ with $w'<w$,
so we can compute all $P_{x,w}$'s with a fixed $\ell(w)$ in parallel.
For technical reasons we haven't been able to implement this parallelization.
\end{enumerate}

As a curious by-product of our computation we get
\begin{corollary}
The values of $\mu(x,w)$ for $x,w\in S_{12}$ are given by
\[
\{1,2,3,4,5,6,7,8,9,10,17,18,19,20,21,22,23,24,25,26,27,28,158,163\}.
\]
\end{corollary}
This complements \cite[Theorem 1.1]{MR2859901}. The new values of $\mu$ are $9,10,17,19,20,21,22$.

\begin{center}
\captionof{table}{Values of $\mu(x,w)$ and pairs attaining them for $S_{12}$}
\[
\begin{array}{c|c|c|c}
w&x&\deg P_{x,w}&\mu(x,w)\\
\hline
35608ab12794 & 032168754ba9 & 6  & 2\\
25789a0b1346 & 0251843a976b & 7  & 3\\
245a6789b013 & 0216543a987b & 8  & 4\\
245789ab0136 & 0215439876ba & 8  & 5\\
35689a0b1247 & 0326541987ba & 8  & 6\\
356a890b1247 & 0326541a987b & 8  & 7\\
36789ab01245 & 0327198654ba & 7  & 8\\
5789a0b12346 & 0524319876ba & 9  & 9\\
789ab1234560 & 210876453ba9 & 10 & 10\\
792b4560a183 & 21076543ba98 & 9  & 17\\
28a3b5670914 & 032187654ba9 & 10 & 18\\
79ab34568012 & 0764321a985b & 11 & 19\\
48a5679b0123 & 054321a9876b & 9  & 20\\
4a56789b0123 & 054321a9876b & 8  & 21\\
46978ab01235 & 0432198765ba & 9  & 22\\
48967ab01235 & 0432198765ba & 10 & 23\\
289a3b456071 & 0432198765ba & 10 & 24\\
46b789a01235 & 043218765ba9 & 11 & 25\\
36789ab01245 & 032187654a9b & 9  & 26\\
3789ab012456 & 032187654a9b & 9  & 27\\
29ab45678013 & 054321a9876b & 10 & 28\\
9ab345678012 & 054321a9876b & 12 & 158\\
89ab34567012 & 054321a9876b & 13 & 163
\end{array}
\]
\end{center}

Complete tables listing the fully reduced pairs in $S_k$, $k\le12$ with $\mu>0$ (together with their $\mu$-value)
are available upon request. The size of the compressed file for $S_{12}$ is 200MB.

Finally, I would like to take this opportunity to thank Amir Gonen,
the Unix system engineer of our faculty, for his technical assistance with running this heavy-duty job.

\def\cprime{$'$}
\providecommand{\bysame}{\leavevmode\hbox to3em{\hrulefill}\thinspace}
\providecommand{\MR}{\relax\ifhmode\unskip\space\fi MR }
\providecommand{\MRhref}[2]{%
  \href{http://www.ams.org/mathscinet-getitem?mr=#1}{#2}
}
\providecommand{\href}[2]{#2}


\begin{thebibliography}{BMB07}

\bibitem[BB05]{MR2133266}
Anders Bj{\"o}rner and Francesco Brenti, \emph{Combinatorics of {C}oxeter
  groups}, Graduate Texts in Mathematics, vol. 231, Springer, New York, 2005.
  \MR{2133266 (2006d:05001)}

\bibitem[BBL98]{MR1620935}
Prosenjit Bose, Jonathan~F. Buss, and Anna Lubiw, \emph{Pattern matching for
  permutations}, Inform. Process. Lett. \textbf{65} (1998), no.~5, 277--283.
  \MR{1620935}

\bibitem[BC17]{MR3558217}
Francesco Brenti and Fabrizio Caselli, \emph{Peak algebras, paths in the
  {B}ruhat graph and {K}azhdan-{L}usztig polynomials}, Adv. Math. \textbf{304}
  (2017), 539--582. \MR{3558217}

\bibitem[BH99]{MR1688445}
Brigitte Brink and Robert~B. Howlett, \emph{Normalizers of parabolic subgroups
  in {C}oxeter groups}, Invent. Math. \textbf{136} (1999), no.~2, 323--351.
  \MR{1688445}

\bibitem[BJS93]{MR1241505}
Sara~C. Billey, William Jockusch, and Richard~P. Stanley, \emph{Some
  combinatorial properties of {S}chubert polynomials}, J. Algebraic Combin.
  \textbf{2} (1993), no.~4, 345--374. \MR{1241505}

\bibitem[BMB07]{MR2376109}
Mireille Bousquet-M\'elou and Steve Butler, \emph{Forest-like permutations},
  Ann. Comb. \textbf{11} (2007), no.~3-4, 335--354. \MR{2376109}

\bibitem[BMS16]{MR3497995}
Francesco Brenti, Pietro Mongelli, and Paolo Sentinelli, \emph{Parabolic
  {K}azhdan-{L}usztig polynomials for quasi-minuscule quotients}, Adv. in Appl.
  Math. \textbf{78} (2016), 27--55. \MR{3497995}

\bibitem[Bor98]{MR1654763}
Richard~E. Borcherds, \emph{Coxeter groups, {L}orentzian lattices, and {$K3$}
  surfaces}, Internat. Math. Res. Notices (1998), no.~19, 1011--1031.
  \MR{MR1654763 (2000a:20088)}

\bibitem[Bre02]{MR1972246}
Francesco Brenti, \emph{Kazhdan-{L}usztig and {$R$}-polynomials, {Y}oung's
  lattice, and {D}yck partitions}, Pacific J. Math. \textbf{207} (2002), no.~2,
  257--286. \MR{1972246}

\bibitem[BW01]{MR1826948}
Sara~C. Billey and Gregory~S. Warrington, \emph{Kazhdan-{L}usztig polynomials
  for 321-hexagon-avoiding permutations}, J. Algebraic Combin. \textbf{13}
  (2001), no.~2, 111--136. \MR{1826948}

\bibitem[BW03]{MR1990570}
\bysame, \emph{Maximal singular loci of {S}chubert varieties in {${\rm
  SL}(n)/B$}}, Trans. Amer. Math. Soc. \textbf{355} (2003), no.~10, 3915--3945.
  \MR{1990570}

\bibitem[BY13]{MR3003920}
Roman Bezrukavnikov and Zhiwei Yun, \emph{On {K}oszul duality for {K}ac-{M}oody
  groups}, Represent. Theory \textbf{17} (2013), 1--98. \MR{3003920}

\bibitem[dC02]{MR1959749}
Fokko du~Cloux, \emph{Computing {K}azhdan-{L}usztig polynomials for arbitrary
  {C}oxeter groups}, Experiment. Math. \textbf{11} (2002), no.~3, 371--381.
  \MR{1959749 (2004j:20084)}

\bibitem[Dem74]{MR0354697}
Michel Demazure, \emph{D\'esingularisation des vari\'et\'es de {S}chubert
  g\'en\'eralis\'ees}, Ann. Sci. \'Ecole Norm. Sup. (4) \textbf{7} (1974),
  53--88, Collection of articles dedicated to Henri Cartan on the occasion of
  his 70th birthday, I. \MR{0354697 (50 \#7174)}

\bibitem[Deo85]{MR788771}
Vinay~V. Deodhar, \emph{Local {P}oincar\'e duality and nonsingularity of
  {S}chubert varieties}, Comm. Algebra \textbf{13} (1985), no.~6, 1379--1388.
  \MR{788771 (86i:14015)}

\bibitem[Deo87]{MR916182}
\bysame, \emph{On some geometric aspects of {B}ruhat orderings. {II}. {T}he
  parabolic analogue of {K}azhdan-{L}usztig polynomials}, J. Algebra
  \textbf{111} (1987), no.~2, 483--506. \MR{916182 (89a:20054)}

\bibitem[Deo90]{MR1065215}
\bysame, \emph{A combinatorial setting for questions in {K}azhdan-{L}usztig
  theory}, Geom. Dedicata \textbf{36} (1990), no.~1, 95--119. \MR{1065215
  (91h:20075)}

\bibitem[Deo94]{MR1278702}
Vinay Deodhar, \emph{A brief survey of {K}azhdan-{L}usztig theory and related
  topics}, Algebraic groups and their generalizations: classical methods
  ({U}niversity {P}ark, {PA}, 1991), Proc. Sympos. Pure Math., vol.~56, Amer.
  Math. Soc., Providence, RI, 1994, pp.~105--124. \MR{1278702 (96d:20039)}

\bibitem[EW14]{MR3245013}
Ben Elias and Geordie Williamson, \emph{The {H}odge theory of {S}oergel
  bimodules}, Ann. of Math. (2) \textbf{180} (2014), no.~3, 1089--1136.
  \MR{3245013}

\bibitem[Fan98]{MR1603806}
C.~K. Fan, \emph{Schubert varieties and short braidedness}, Transform. Groups
  \textbf{3} (1998), no.~1, 51--56. \MR{1603806 (98m:14052)}

\bibitem[FG97]{MR1441960}
C.~K. Fan and R.~M. Green, \emph{Monomials and {T}emperley-{L}ieb algebras}, J.
  Algebra \textbf{190} (1997), no.~2, 498--517. \MR{1441960 (98a:20037)}

\bibitem[Fox]{Fox}
Jacob Fox, \emph{{S}tanley-{W}ilf limits are typically exponential}, Adv. Math.
  \textbf{to appear}, arXiv:1310.8378.

\bibitem[GM14]{MR3376367}
Sylvain Guillemot and D\'aniel Marx, \emph{Finding small patterns in
  permutations in linear time}, Proceedings of the {T}wenty-{F}ifth {A}nnual
  {ACM}-{SIAM} {S}ymposium on {D}iscrete {A}lgorithms, ACM, New York, 2014,
  pp.~82--101. \MR{3376367}

\bibitem[Hen07]{MR2320806}
Anthony Henderson, \emph{Nilpotent orbits of linear and cyclic quivers and
  {K}azhdan-{L}usztig polynomials of type {A}}, Represent. Theory \textbf{11}
  (2007), 95--121 (electronic). \MR{2320806}

\bibitem[How80]{MR576184}
Robert~B. Howlett, \emph{Normalizers of parabolic subgroups of reflection
  groups}, J. London Math. Soc. (2) \textbf{21} (1980), no.~1, 62--80.
  \MR{576184}

\bibitem[JW13]{MR3027577}
Brant Jones and Alexander Woo, \emph{Mask formulas for cograssmannian
  {K}azhdan-{L}usztig polynomials}, Ann. Comb. \textbf{17} (2013), no.~1,
  151--203. \MR{3027577}

\bibitem[KL79]{MR560412}
David Kazhdan and George Lusztig, \emph{Representations of {C}oxeter groups and
  {H}ecke algebras}, Invent. Math. \textbf{53} (1979), no.~2, 165--184.
  \MR{560412 (81j:20066)}

\bibitem[KT02]{MR1901161}
Masaki Kashiwara and Toshiyuki Tanisaki, \emph{Parabolic {K}azhdan-{L}usztig
  polynomials and {S}chubert varieties}, J. Algebra \textbf{249} (2002), no.~2,
  306--325. \MR{1901161 (2004a:14049)}

\bibitem[Lap17]{1710.06115}
Erez Lapid, \emph{A tightness property of relatively smooth permutations},
  2017, arXiv:1710.06115.

\bibitem[Las95]{MR1354702}
Alain Lascoux, \emph{Polyn\^omes de {K}azhdan-{L}usztig pour les vari\'et\'es
  de {S}chubert vexillaires}, C. R. Acad. Sci. Paris S\'er. I Math.
  \textbf{321} (1995), no.~6, 667--670. \MR{1354702 (96g:05144)}

\bibitem[LM16]{1605.08545}
Erez Lapid and Alberto M{\'{\i}}nguez, \emph{Geometric conditions for
  $\square$-irreducibility of certain representations of the general linear
  group over a non-archimedean local field}, 2016, arXiv:1605.08545.

\bibitem[LS90]{MR1051089}
V.~Lakshmibai and B.~Sandhya, \emph{Criterion for smoothness of {S}chubert
  varieties in {${\rm Sl}(n)/B$}}, Proc. Indian Acad. Sci. Math. Sci.
  \textbf{100} (1990), no.~1, 45--52. \MR{1051089}

\bibitem[Lus93]{MR1261904}
G.~Lusztig, \emph{Tight monomials in quantized enveloping algebras}, Quantum
  deformations of algebras and their representations ({R}amat-{G}an, 1991/1992;
  {R}ehovot, 1991/1992), Israel Math. Conf. Proc., vol.~7, Bar-Ilan Univ.,
  Ramat Gan, 1993, pp.~117--132. \MR{1261904}

\bibitem[Lus03]{MR1974442}
\bysame, \emph{Hecke algebras with unequal parameters}, CRM Monograph Series,
  vol.~18, American Mathematical Society, Providence, RI, 2003. \MR{1974442}

\bibitem[Lus77]{MR0453885}
\bysame, \emph{Coxeter orbits and eigenspaces of {F}robenius}, Invent. Math.
  \textbf{38} (1976/77), no.~2, 101--159. \MR{0453885}

\bibitem[LW17]{1702.00459}
Nicolas Libedinsky and Geordie Williamson, \emph{The anti-spherical category},
  2017, arXiv:1702.00459.

\bibitem[Mac04]{MR2417935}
Percy~A. MacMahon, \emph{Combinatory analysis. {V}ol. {I}, {II} (bound in one
  volume)}, Dover Phoenix Editions, Dover Publications, Inc., Mineola, NY,
  2004, Reprint of {\it An introduction to combinatory analysis} (1920) and
  {\it Combinatory analysis. Vol. I, II} (1915, 1916). \MR{2417935}

\bibitem[Mon14]{MR3159260}
Pietro Mongelli, \emph{Kazhdan-{L}usztig polynomials of {B}oolean elements}, J.
  Algebraic Combin. \textbf{39} (2014), no.~2, 497--525. \MR{3159260}

\bibitem[MT04]{MR2063960}
Adam Marcus and G\'abor Tardos, \emph{Excluded permutation matrices and the
  {S}tanley-{W}ilf conjecture}, J. Combin. Theory Ser. A \textbf{107} (2004),
  no.~1, 153--160. \MR{2063960}

\bibitem[Sen14]{MR3166061}
Paolo Sentinelli, \emph{Isomorphisms of {H}ecke modules and parabolic
  {K}azhdan-{L}usztig polynomials}, J. Algebra \textbf{403} (2014), 1--18.
  \MR{3166061}

\bibitem[SW04]{MR2043806}
Zvezdelina Stankova and Julian West, \emph{Explicit enumeration of 321,
  hexagon-avoiding permutations}, Discrete Math. \textbf{280} (2004), no.~1-3,
  165--189. \MR{2043806}

\bibitem[Ten07]{MR2333139}
Bridget~Eileen Tenner, \emph{Pattern avoidance and the {B}ruhat order}, J.
  Combin. Theory Ser. A \textbf{114} (2007), no.~5, 888--905. \MR{2333139}

\bibitem[War11]{MR2859901}
Gregory~S. Warrington, \emph{Equivalence classes for the {$\mu$}-coefficient of
  {K}azhdan-{L}usztig polynomials in {$S_n$}}, Exp. Math. \textbf{20} (2011),
  no.~4, 457--466. \MR{2859901 (2012j:05460)}

\bibitem[Wes96]{MR1417303}
Julian West, \emph{Generating trees and forbidden subsequences}, Proceedings of
  the 6th {C}onference on {F}ormal {P}ower {S}eries and {A}lgebraic
  {C}ombinatorics ({N}ew {B}runswick, {NJ}, 1994), vol. 157, 1996,
  pp.~363--374. \MR{1417303}

\bibitem[Yun09]{MR2480718}
Zhiwei Yun, \emph{Weights of mixed tilting sheaves and geometric {R}ingel
  duality}, Selecta Math. (N.S.) \textbf{14} (2009), no.~2, 299--320.
  \MR{2480718}

\end{thebibliography}
\end{document}